\documentclass[twoside,12pt,reqno]{amsart}
\usepackage{amsmath,amscd,amssymb,epsfig,fullpage,subfig} 
\usepackage[all]{xy}
\usepackage{verbatim}
\usepackage{dsfont}

\newcommand{\End}{\operatorname{End}}
\newcommand{\Hom}{\operatorname{Hom}}
\newcommand{\Id}{\operatorname{Id}}
\newcommand{\tr}{\operatorname{tr}}
\newcommand{\Tr}{\operatorname{Tr}}

\newcommand{\Z}{\mathbb{Z}}
\newcommand{\R}{\mathbb{R}}
\newcommand{\C}{\mathbb{C}}
\newcommand{\N}{\mathbb{N}}
\newcommand{\sll}{\mathfrak{sl}}

\newcommand{\g}{\ensuremath{\mathfrak{g}}}
\newcommand{\p}[1]{\ensuremath{\overline {#1}}}
\newcommand{\md}{\operatorname{\mathsf{d}}}
\newcommand{\mt}{\operatorname{\mathsf{t}}}
\newcommand{\fg}{\g}
\renewcommand{\dim}{\ensuremath{\operatorname{dim}}}
\newcommand{\Rad}{\operatorname{Rad}}

\newcommand{\ideal}{\mathcal{I}}
\newcommand{\cat}{\mathcal{C}}
\newcommand{\ob}{\operatorname{Ob}(\cat)}
\newcommand{\obideal}{\operatorname{Ob}(\ideal_{J})}
\newcommand{\unit}{\ensuremath{\mathds{1}}}
\newcommand{\glmn}{\ensuremath{\mathfrak{gl}(m|n)}}
\newcommand{\Proj}{\ensuremath{\mathcal{P}roj}}
\newcommand{\atyp}{\operatorname{atyp}}
\newcommand{\fe}{\mathfrak{e}}
\newcommand{\defect}{\operatorname{def}}
\newcommand{\sdim}{\operatorname{sdim}}
\newcommand{\rank}{\operatorname{rank}}
\newcommand{\fb}{\ensuremath{\mathfrak{b}}}
\newcommand{\ft}{\ensuremath{\mathfrak{t}}}
\newcommand{\fa}{\ensuremath{\mathfrak{a}}}
\newcommand{\0}{\bar 0}
\renewcommand{\1}{\bar 1}
\newcommand{\V}{\mathcal{V}}
\newcommand{\epsh}[2]
         {\begin{array}{c} \hspace{-1.3mm}
        \raisebox{-4pt}{\epsfig{figure=#1,height=#2}}
        \hspace{-1.9mm}\end{array}}

\newtheorem{definition}{Definition}[subsection]
\newtheorem{theorem}[definition]{Theorem}

\newtheorem{proposition}[definition]{Proposition}
\newtheorem{lemma}[definition]{Lemma}
\newtheorem{remark}[definition]{Remark}
\newtheorem{example}[definition]{Example}
\newtheorem{corollary}[definition]{Corollary}
\newtheorem{conjecture}[definition]{Conjecture}

\newcommand{\pic}[2]{
  \setlength{\unitlength}{#1}
  {\begin{array}{c} \hspace{-1.3mm}
        \raisebox{-4pt}{#2}
        \hspace{-1.9mm}\end{array}}}

\newcommand{\TrR}{
          \qbezier(4, 3)(4, 0)(6, 0)
          \qbezier(6, 0)(9, 0)(9, 4)
          \put(9,4){\vector(0,1)3}
          \multiput(0,3)(6,0){2}{\line(0,1){5}}
          \multiput(0,3)(0,5){2}{\line(1,0){6}}
          \qbezier(4, 8)(4, 11)(6, 11)
          \qbezier(6, 11)(9, 11)(9, 7)
          \put(2,11){\line(0,-1){3}}
          \put(2,3){\vector(0,-1){3}}
        }
\newcommand{\TrL}{
          \qbezier(5, 3)(5, 0)(3, 0)
          \qbezier(3, 0)(0, 0)(0, 4)
          \put(0,4){\vector(0,1)3}
          \multiput(9,3)(-6,0){2}{\line(0,1){5}}
          \multiput(3,3)(0,5){2}{\line(1,0){6}}
          \qbezier(5, 8)(5, 11)(3, 11)
          \qbezier(3, 11)(0, 11)(0, 7)
          \put(7,11){\line(0,-1){3}}
          \put(7,3){\vector(0,-1){3}}
        }

\numberwithin{equation}{subsection}

\begin{document}
\title{Generalized trace and modified dimension functions on ribbon categories} 
  \date{\today}
\author{Nathan Geer}
\address{Mathematics \& Statistics\\
  Utah State University \\
  Logan, Utah 84322, USA
and\\
Max-Planck-Institut f\"ur Mathematik\\
Vivatsgasse 7\\
53111 Bonn, Germany}
\thanks{Research of the first author was partially supported by NSF grants
DMS-0706725 and DMS-0968279}\
\email{nathan.geer@usu.edu}
\author{Jonathan Kujawa}
\address{Mathematics Department\\
University of Oklahoma\\
Norman, OK 73019}
\thanks{Research of the second author was partially supported by NSF grant
DMS-0734226}\
\email{kujawa@math.ou.edu}
\author{Bertrand Patureau-Mirand}
\address{LMAM, Universit\'e de Bretagne-Sud, BP 573\\
F-56017 Vannes, France}
\email{bertrand.patureau@univ-ubs.fr}

\begin{abstract}
In this paper we use topological techniques to construct generalized trace and modified dimension functions on ideals in certain ribbon categories.  Examples of such ribbon categories naturally arise in representation theory where the usual trace and dimension functions are zero, but these generalized trace and modified dimension functions are non-zero.  Such examples include categories of finite dimensional modules of certain Lie algebras and finite groups over a field of positive characteristic and categories of finite dimensional modules of basic Lie superalgebras over the complex numbers.  These modified dimensions can be interpreted categorically and are closely related to some basic notions from representation theory. 
\end{abstract}

\maketitle
\setcounter{tocdepth}{1}

\section{Introduction}

\subsection{}\label{SS:jibberjabber} 
It is well understood that there is a strong connection between representation theory and low dimensional topology. The path is especially well trodden going from algebra to topology (although
notable exceptions do exist \cite{BTT}).  In this paper we use the topological techniques of \cite{GPT} to define generalized  trace and modified dimension functions for certain ribbon categories and illustrate this theory with several examples arising in representation theory.  In these examples the usual trace and dimension functions are trivial, however the generalized trace and modified dimension functions are non-trivial and are closely related to the underlying representation theory.  

For our purposes we have in mind the
approach in \cite{Tu}.  The general idea is to start with some suitable
category (e.g. finite dimensional representations of some
algebraic object) which admits a tensor product and braiding isomorphisms
\[
c_{V,W}: V \otimes W \xrightarrow{\cong} W \otimes V
\]
for all $V$ and $W$ in the category, and then use that category to create
invariants of knots, links, 3-manifolds, etc.\ by interpreting the relevant
knot or link as a morphism in the category using the braiding to represent
crossings in the knot or link diagram.  If you apply this machine you quickly
discover two difficulties:
\begin{enumerate}
\item Many categories arising in algebra are symmetric (i.e.\ the square of
  the braiding is the identity) and, hence, yield only trivial invariants.
\item Many objects in these categories have categorical dimension zero and,
  again, necessarily yield only trival invariants.
\end{enumerate}  

Regarding the first, many natural categories fit within the above framework but happen to be symmetric (e.g. representations of finite groups, Lie algebras, etc.).  Therefore they have not received the same level of study from this point of view as, for example, representations of quantum groups.  A priori, however, there is no reason to exclude symmetric categories from consideration.  Indeed, we will see that the symmetric categories are equally interesting.  

Tackling the second problem, the first and third authors introduced a modified dimension for representations of quantum groups associated to Lie superalgebras \cite{GP2}, and with Turaev a generalization of this construction to include, for example, the quantum group for $\sll(2)$ at a root of unity \cite{GPT}.  The key point is that this modified dimension has properties analogous to the categorical dimension, but can still be nonzero even when the categorical dimension is zero.  This allows one to apply the above machine to objects with  categorical dimension zero and obtain nontrivial topological invariants.  The main goal of this paper is put these results into a categorical framework which includes the non-quantum setting and to provide several examples and applications within well studied categories in representation theory.  

It should also be noted that in \cite{GP4} the first and third authors constructed generalized trace and modified dimension functions on the ideal of projective modules in the category of finite dimensional modules over a Lie superalgebra of Type A or C.   The techniques of \cite{GP4} are based on quantum groups and the Kontsevich integral.    The work of this paper is a vast generalization of \cite{GP4} and is based on purely topological techniques.

\subsection{}\label{SS:mainresults}   The basic setting of our results is within a ribbon category $\cat$.  That is, roughly speaking, within a category $\cat$ with a tensor product bifunctor 
\[
- \otimes - : \cat  \times \cat  \to \cat, 
\] a unit object $\unit$, and a braiding; that is, for all $V$ and $W$ in $\cat$ we have canonical isomorphisms 
\[
c_{V,W}: V \otimes W \to W \otimes V.
\] Furthermore, $\cat$ admits a duality functor 
\[
V\mapsto V^{*},
\] and morphisms 
\begin{align*}
b_{V}: \unit \to V \otimes V^{*} & & d_{V}: V^{*}\otimes V \to \unit , \\
b'_{V}: \unit \to V^{*} \otimes V & & d'_{V}: V\otimes V^{*} \to \unit. 
\end{align*} This data is subject to suitable axioms.  Section~\ref{SS:ribboncats} for the precise definition.  

Such categories are ubiquitus in nature.  Let us mention just a few occurrances.  In algebra, examples of such categories include finite dimensional representations of groups, Lie (super)algebras, and quantum groups.  In algebraic geometry, they arise as the derived categories of perfect complexes over certain schemes \cite{Ba}, and in the theory of motives as Tannakian categories.  In topology they can be found in stable homotopy theory \cite{HPS} and, as discussed above, are an integral part of low dimensional topology. Ribbon categories also arise as examples of fusion categories \cite{ENO}.  On the other hand, ribbon categories subject to additional axioms are the source of topological quantum field theories.  These in turn give 3-manifold invariants \cite{RT}, and provides connections to physics \cite{Wi1,Wi2} and quantum computing \cite{FKW,FKLW}.

\subsection{}\label{SS:theory}  The paper has two main components.  In the first half (Sections~\ref{S:generalizedtraces}-\ref{S:ComRing}) we introduce the fundamental new concepts of the paper.  Let $\cat$ be a ribbon category and set $K=\End_{\cat}(\unit )$.  Throughout we assume $\cat$ is an Ab-category; i.e.\ for all $V,W$ in $\cat$, $\Hom_{\cat}(V,W)$ is an additive abelian group and both the tensor product functor and composition of morphisms are bilinear.  We also assume for convenience that $K$ is a field.  Neither assumption is particularly strict.  Even so, they are stricter than is necessary.  In particular, we discuss in Section~\ref{S:ComRing} how to drop the assumption that $K$ is a field.    

We call a full subcategory $\ideal$ of $\cat$ an \emph{ideal} if it is closed under retracts (i.e.\ if $W \in \ideal$ and $\alpha: X \to W$ and $\beta: W \to X$ satisfies $\beta \circ \alpha = \Id_{X}$, then $X \in \ideal$) and if $X$ in $\cat$ and $Y$ in $\ideal$ implies $X \otimes Y$ is in $\ideal$.   We will be most interested in the case when one fixes an object $J$ in $\cat$ and takes $\ideal_{J}$ to be the ideal of all objects which are retracts of $V \otimes X$ for some $X$ in $\cat$.

Our first new definition is as follows.  A \emph{trace} on an ideal $\ideal$ is a family of $K$-linear functions $\mt = \{\mt_{V} \}_{V \in \ideal}$
\[
\mt_{V}: \End_{\cat}(V) \to K,
\] which is suitably compatible with the tensor product and composition of morphisms (cf.\ Definition~\ref{D:trace}).  Given a trace on $\ideal$, $\mt=\{\mt_{V} \}_{V \in \ideal}$, we can then define a \emph{modified dimension function on $\ideal$} via
\begin{equation}\label{E:moddimfunction}
\md_{\mt}(V)=\mt_{V}(\Id_V).
\end{equation}

Our second new definition is as follows.  Assume that $J$ in $\cat$ admits a linear map 
\[
\mt_{J}: \End_{\cat}(J) \to K
\] which satisfies 
\begin{equation*}
\mt_{J}\left(  \left(d_{J}\otimes \Id_{J})\circ(\Id_{J^{*}}\otimes
h)\circ(b'_{J}\otimes \Id_{J} \right) \right) = \mt_{J} \left(  \left(\Id_{J}\otimes d'_{J}) \circ (h \otimes \Id_{J^{*}})\circ(\Id_{J}\otimes b_{J} \right)\right),
\end{equation*} for all $h \in \End_{\cat}(J\otimes J)$.  That is, in the
graphical calculus discussed in Section~\ref{SS:diagrams} we have  
$$\quad \mt_J\left(\pic{0.6ex}{
\begin{picture}(10,11)(1,0)
  \TrR\put(2,4.5){$h$}
\end{picture}}\right)=\mt_J\left(\pic{0.6ex}{
\begin{picture}(10,11)(1,0)
  \TrL\put(5,4.5){$h$}
\end{picture}}\right)$$
for all $h\in \End_{\cat}(J\otimes J)$.
Such a linear map is called an \emph{ambidextrous trace on $J$}. 

Our first main result shows that these two notions are intimately related.  Recall that $\ideal_{J}$ is the ideal whose objects are those which are retracts of $J \otimes X$ for some $X$ in $\cat$.
\begin{theorem} If $\ideal$ is an ideal of a ribbon category $\cat$ and
  $\{\mt_{V} \}_{V \in \ideal}$ is a trace on $\ideal$, then each $\mt_{V}$ is an
  ambidextrous trace on $V$.  Conversely, if $J$ in $\cat$ admits an
  ambidextrous trace, then there is a unique trace on $\ideal_{J}$ determined
  by that ambidextrous trace.
\end{theorem} 

\subsection{} We will be particularly interested when there is a canonical choice for a trace function.  Namely, assume $J$ is absolutely indecomposable (i.e.\ $\End_{\cat}(J)/\Rad \left( \End_{\cat}(J) \right) \cong K$) and the canonical projection
\[
t: \End_{\cat}(J) \to \End_{\cat}(J)/\Rad \left( \End_{\cat}(J) \right) \cong K
\] is an ambidextrous trace on $J$.   An absolutely indecomposable object whose canonical map gives an ambidextrous trace is called \emph{ambidextrous}.  Using the trace on
$\ideal_{J}$ defined by the previous theorem we use \eqref{E:moddimfunction} to define a modified dimension
function
\[
\md_{J}: \obideal \to K
\] by 
\[\md_{J}(V) = \mt_{V}\left(\Id_{V} \right).
\]

Applying this construction in the special case when $J=\unit$, the identity map $\End_{\cat}(\unit ) \to K$ defines an ambidextrous trace on $\ideal_{\unit} = \cat$ and we recover the familiar notions of categorical trace and dimension.  Thus the above setup generalizes these well studied functions.
We also see that results known for the categorical dimension hold for the modified dimension as well.  For example, we have the following results.
\begin{theorem}\label{T:dimprops}  Let $\cat$ be an abelian category, $J$ be ambidextrous and let $V$ be an object in $\ideal_{J}$ with $\End_{\cat}(V)=K$ and $d_{V}:V^{*}\otimes V \to \unit$ an epimorphism.  We then have the following results.
\begin{enumerate}
\item Let $ U \in \ideal_{V} \subseteq \ideal_J$.  If $\md_J(V)= 0$, then $\md_J(U)=0$.
\item The canonical epimorphism $d_{V}\otimes \Id_{J}: V^*\otimes V \otimes J\xrightarrow{} J\rightarrow 0$ splits if and only if $\md_J(V)\neq 0$. 
\item If $J$ is not projective in $\cat$ and $P$ is projective in $\cat$, then $P$ is an object of $\ideal_{J}$ and $\md_{J}(P)=0$.
\end{enumerate}
\end{theorem}  The above results may become more recognizable when we specialize to the particular case of finite dimensional representations of a finite group over an algebraically closed field of characteristic $p$ and take $J$ to be the trivial module.  In this setting the first statement of the theorem becomes the statement that if $p$ divides the dimension of $V$, then $p$ divides the dimension of any direct summand of $V \otimes X$ for any module $X$.  The second statement becomes the statement that the trivial module is a direct summand of $V^{*}\otimes V$ if and only if $p$ does not divide the dimension of $V$.  In this particular context these results were proven by Benson and Carlson \cite{BC}.  The third statement becomes the well known result that for a finite group whose order is divisible by $p$, the projective modules over a field of characteristic $p$ all have dimension divisible by $p$.  As another example, if we specialize Theorem~\ref{T:dimprops}(1) to when $\cat$ is the finite dimensional representations of a quantum group at a root of unity and again $J$ is the trivial module, we then recover a result of Andersen \cite[Lemma 3.6]{An}.  The above theorem demonstrates that these results fit within a more general categorical framework.

\subsection{}\label{SS:practice} 
  The second half of
the paper (Sections~\ref{S:Liesuperalgebras}-\ref{S:sl2}) is devoted to
applying the above theory to specific settings.  The examples were chosen
primarily on the basis of the areas of expertise among the authors.  It would
also be interesting to investigate the theory in other contexts.

We first consider the finite dimensional representatations of a basic classical Lie superalgebra $\fg$ over the complex numbers, $\C$ (see Section~\ref{S:Liesuperalgebras} for definitions). We prove that when $\fg = \glmn$ or a simple Lie superalgebra of type A or C, then the modified dimension defined categorically here coincides with the one defined in \cite{GP2} using supercharacters.  As a conseqence we obtain an explicit formula for $\md_{J}(L)$ whenever both $J$ and $L$ are simple supermodules of atypicality zero.  Furthermore, the formula implies it is nonzero.  In contrast, the categorical dimension is zero for such supermodules.  It was precisely to avoid problem (2) discussed above which lead the authors of \cite{GP2} to their original formulation of the modified dimension (in the quantum setting). 

In the setting of basic classical Lie superalgebras Kac and Wakimoto \cite{KW} introduced combinatorially defined integers called the defect of $\fg$ and the atypicality of a simple supermodule of $\fg$.  Let us write $\defect (\fg )$ for the defect of $\fg$ and $\atyp (L)$ for the atypicality of a simple supermodule $L$.  In general, the atypicality of a simple supermodule is among $0, 1, 2, \dotsc , \defect(\fg )$ and $\defect (\fg ) = \atyp (\C )$, where $\C$ is the trivial module.  Also, recall that if $L=L_{\0}\oplus L_{\1}$ is a supermodule, then the categorical dimension is given by the superdimension: $\sdim (L) = \dim_{\C}\left(L_{\0} \right)-\dim_{\C}\left(L_{\1} \right)$.  Kac and Wakimoto stated the following intriguing conjecture \cite[Conjecture 3.1]{KW}.
\begin{conjecture}  Let $\fg$ be a basic classical Lie superalgebra and let $L$ be a simple $\fg$-supermodule.  Then 
\[
\atyp (L) =\defect(\fg)  \text{ if and only if } \sdim (L)\neq 0. 
\]
\end{conjecture}  Partial results are known (e.g. \cite{BKN1, DS, KW}) and recently Serganova has announced a proof for the classical contragradiant Lie superalgebras using category equivalences, Zuckerman functors, and a character formula of Penkov \cite{Ser}.  In any case, our framework immediately suggests that this conjecture is but the ``top level''  (that is, when $J=\C$ and  $\atyp (J) = \defect (\fg )$) of the following generalized Kac-Wakimoto conjecture.

\begin{conjecture}
 Let $\fg$ be a basic classical Lie superalgebra, let $J$ be a simple $\fg$-supermodule which admits an ambidextrous trace and let $L \in \ideal_{J}$ be a simple $\fg$-supermodule.  Then 
\[
\atyp(L) = \atyp(J) \text{ if and only if } \md_J(L) \neq 0.
\]
\end{conjecture}

In the case of $\glmn$ we can provide the following strong evidence for the generalized Kac-Wakimoto conjecture.  We remind the reader that a simple $\fg$-supermodule is, by definition, polynomial if it appears as a composition factor of some tensor power of the natural module. 

\begin{theorem}\label{Intro:divisible}  Let $\fg =\glmn$, let $J$ be a simple $\fg$-supermodule which admits an ambidextrous trace, and let $L \in \ideal_{J}$ be a simple $\fg$-supermodule.  Then the following are true.
\begin{enumerate}
\item One always has $\atyp (L) \leq \atyp (J)$. 
\item If $\md_{J}(L) \neq 0$, then $\atyp (L) = \atyp (J)$.
\item If $\atyp (J)=0,$ then $\atyp (L) = \atyp (J)$ and $\md_{J}(L) \neq 0$.
\item If $J$ and $L$ are polynomial, then $J$ necessarily admits an ambidextrous trace (i.e.\ it does not have to be assumed), and $\md_{J}(L) \neq 0$ if and only if $\atyp (L) = \atyp (J)$.
\end{enumerate}
\end{theorem}  That is, for $\glmn$ we can prove one direction of the generalized Kac-Wakimoto conjecture in general, and both directions for both atypicality zero and polynomial representations.

\subsection{}  We also examine the case when $\cat$ is the finite dimensional representations of a finite group $G$ over an algebraically closed field.  We consider the cases when $G$ is the cyclic group of order $p$ over a field of characteristic $p$, and the Klein four group over a field of characteristic two.  We use explicit calculations and the results of earlier authors to analyze the ideal structure of $\cat$ (cf.\ Proposition~\ref{P:kleingroup}) and to prove the existence of indecomposable modules whose canonical trace is ambidextrous.  We prove that $\cat$ has ambidextrous objects and, perhaps most intriguingly, see by direct calculation that a certain family of two dimensional indecomposable modules of the Klein four group rather unexpectedly admits an ambidextrous trace. 

We also consider the case when $\cat$ is the finite dimensional representations of the Lie algebra $\sll_{2}(k)$ over an algebraically closed field of characteristic $p>2$.  In this case we focus on the simple $\sll_{2}(k)$-modules.  We analyze the ideal structure of $\cat$ (Theorem~\ref{T:sl2ideals}) and obtain the following complete classification of which simple $\sll_{2}(k)$-modules admit an ambidextrous trace.

\begin{theorem}\label{Intro:sl2mods} A simple $\sll (2)$-module admits a nontrivial ambidextrous trace if and only if it is either restricted or projective.
\end{theorem}

\subsection{}  The results of this paper raise a number of intriguing questions.  For example, it remains mysterious which objects and ideals in a category admit a nontrivial ambidextrous trace.  Further examples need to be developed to shed light on this question.   

As another example, the generalized Kac-Wakimoto conjecture and the results of Theorem~\ref{Intro:divisible} suggest a close relationship between the modified dimension and classical results on the vanishing of the categorical dimension.  This includes the well known Kac-Weisfeiler conjecture for Lie algebras in characteristic $p$ (proved by Premet in \cite{Pr2}), the DeConcini, Kac and Procesi conjecture for quantum groups at a root of unity \cite{Kr}, $2$ and $p$ divisibility for Lie superalgebras \cite{BKN2} and \cite{WZ}, and well known $p$ divisibility results for modular representations of finite groups.  In many of these contexts one has the powerful tool of support varieties.   The results just mentioned and the theory presented here are both compatible with these support variety theories and it would be interesting to further develop this relationship.

In a third direction, the category of ribbon graphs is naturally a $2$-category.  The notion of $2$-categories have recently received a great deal of attention in representation theory due to work of Khovanov and Lauda \cite{KL1,KL2, KL3, KL4} and Rouquier \cite{Ro} in their study of categorification of quantum groups and their representations.  Of particular relevance to the work here, Ganter and Kapranov categorified the notion of the categorical trace and used this to study representations and character theory in $2$-categories \cite{GK}.  It would be interesting to determine a categorification of our more general notion of a trace. 

Finally, we recall that one can define the radical of a ribbon category $\cat$ as a certain ideal defined by the categorical trace.  The resulting quotient category plays an important role in representation theory. For example, Andersen constructed a three dimensional quantum field theory from the category of tilting modules for a quantum group at a root of unity via this technique \cite{An}.  In recent work Deligne \cite{De} and Knop \cite{Kn} used this approach to show how to construct categories which interpolate among the representation categories of the symmetric groups and $GL(n, \mathbb{F}_{q})$, respectively.  In a similar fashion if one has a trace on an ideal in the sense of this paper one can define the ``radical'' of the ideal using this trace.  We would expect that our construction would allow one to refine the above technique by allowing one to consider subquotient categories of $\cat$.


\section{Ribbon Ab-categories and the graphical calculus}\label{S:ribboncats}

In this section we provide the framework within which the results of this paper are developed.  In Section~\ref{SS:ribboncats} we introduce the notion of a ribbon Ab-category.  Many familar categories in representation theory (e.g.\ finite dimensional representations of Lie (super)algebras, groups, and quantum groups) are ribbon Ab-categories.  A key feature of ribbon Ab-categories is the ability to represent morphisms via diagrams.  Manipulations of the diagrams correspond to identities among morphisms and this provides a powerful tool for understanding morphisms in $\cat$.  We provide a brief overview of this graphical calculus in Section~\ref{SS:diagrams}.

\subsection{Ribbon Ab-categories}\label{SS:ribboncats}
For notation and the general setup of ribbon Ab-categories our references are \cite{Tu} and \cite{Kas}.  A \emph{tensor category} $\cat$ is a category equipped with a covariant
bifunctor $\otimes :\cat \times \cat\rightarrow \cat$ called the tensor product, a unit object $\unit$, an associativity constraint, and left and
right unit constraints such that the Triangle and Pentagon Axioms hold (see \cite[XI.2]{Kas}).  In particular, for any $V$ in $\cat$, $\unit \otimes V$ and $V \otimes \unit$ are canonically isomorphic to $V$.

A \emph{braiding} on a tensor category $\cat$ consists of a family of isomorphisms $\{c_{V,W}: V \otimes W \rightarrow W\otimes V \}$, defined for each pair of objects $V,W$ which satisfy the Hexagon Axiom \cite[XIII.1 (1.3-1.4)]{Kas} as well as the naturality condition expressed in the commutative diagram \cite[(XIII.1.2)]{Kas}.
We say a tensor category is \emph{braided} if it has a braiding.  We call a tensor category \emph{symmetric} if $c_{W, V} \circ c_{V,W} = \Id_{V \otimes W}$ for all $V$ and $W$ in $\cat$.

A tensor category $\cat$ has \emph{duality} if for each object $V$ in $\cat$ there exits an object $V^{*}$ and morphisms 
$$b_{V}: \unit \rightarrow V\otimes V^{*} \: \text{ and } \: d_{V}: V^{*}\otimes V \rightarrow \unit$$
satisfying relations
\begin{align}\label{E:defduality}
(\Id_V\otimes d_V)\circ (b_V\otimes \Id_V)&=\Id_V \\ \notag
 (d_V \otimes \Id_{V^*}) \circ (\Id_{V^*}\otimes b_V) & =\Id_{V^*}.
\end{align}
A \emph{twist} in a braided tensor category $\cat$ with duality is a family $\{ \theta_{V}:V\rightarrow V \}$ of natural isomorphisms defined for each object $V$ of $\cat$ satisfying relations \cite[(XIV.3.1-3.2)]{Kas}. Let us point out that the existence of twists is equivalent to having functorial isomorphisms $V \xrightarrow{\cong} V^{**}$ for all $V$ in $\cat$ (cf.\ \cite[Section 2.2]{BK}).

A \emph{ribbon category} is a braided tensor category with duality and a twist.   
A tensor category $\cat$ is said to be an \emph{Ab-category} if for any pair
of objects $V,W$ of $\cat$ the set of morphism $\Hom_{\cat }(V,W)$ is an additive
abelian group and the composition and tensor product of morphisms are
bilinear.  

Let us end this section with two useful observations about ribbon Ab-categories.  The first is that if $\cat$ is an abelian category, then by \cite[Proposition 2.1.8]{BK} the tensor product functor is necessarily exact in both entries.  The second is that any symmetric tensor category with duality is necessarily a ribbon category \cite[Corollary 2.2.3]{BK}.  As a consequence, many categories which arise in representation theory  are ribbon Ab-categories.  

\subsection{The ground ring of $\cat$}\label{SS:groundring} Let $\cat$ be a ribbon Ab-category. Composition of morphisms
induces a commutative ring structure on 
\[
K=\End_{\cat }(\unit).
\] This ring is called
the \emph{ground ring} of $\cat$. In this paper we will assume for convenience that $K$ is a field.  However, the setup and theorems of the first three sections are valid even when $K$ is an arbitrary commutative ring.  In Section \ref{S:ComRing} we discuss which changes are needed in order to achieve this generality.  

 We note that for any pair of objects
$V$ and $W$ of $\cat$ the abelian group $\Hom_{\cat }(V,W)$ becomes a left $K$-module. 
Namely, for any $k\in K$ and $f\in \Hom_{\cat}(V,W)$ the action is defined by $kf=k\otimes f$ and using the left and right unit constraints.

\subsection{Absolutely irreducible and Indecomposable objects}  An object $V$ of $\cat$ is called \emph{absolutely irreducible} if $\End_{\cat}(V)= K \Id_V$.   We say it is \emph{absolutely indecomposable} if 
\[
\End_{\cat}(V)/\Rad\left( \End_{\cat}(V) \right) \cong K
\] 
where $\Rad(\End_\cat(V))$ is the radical of $\End_\cat(V)$.  In either case we write 
\begin{equation}\label{E:canonicalmap}
\langle \; \rangle : \End_{\cat}(V) \to K
\end{equation}
for the canonical linear map.   

\emph{Throughout we assume that if $J$ is absolutely indecomposable, then the elements of the radical of $\End_{\cat}(J)$ are nilpotent.}  This is not a very restrictive assumption.  For example, if $\End_{\cat}(J)$ is artinian (e.g.\ if $J$ is of finite length), then every element of the radical of $\End_{\cat}(J)$ is nilpotent.  If the reader prefers, all statements involving an absolutely indecomposable object can be specialized to the assumption that $J$ is absolutely irreducible and then no extra assumptions are required.

\subsection{The categorical trace and dimension}\label{SS:categoricaltrace}  For brevity and convenience we define following morphisms in $\cat$,
\begin{align*}
 b'_{V} &: 
  \unit\rightarrow V^*\otimes V, \\
   d_{V}' &: V\otimes V^{*}\rightarrow \unit,
\end{align*}
given by 
\begin{align*}
b'_V &=(\Id_{V^*}\otimes \theta_V) \circ c_{V,V^*} \circ b_V \\
d'_V & = d_V \circ c_{V,V^*} \circ (\theta_V\otimes \Id_{V^*})
\end{align*}

Then for any $V$ in $\cat$ and $f \in \End_{\cat}(V)$,  the \emph{categorical trace on $\cat$} is given by 
$$\tr_{\cat} (f)= d'_{V} \circ ( f \otimes \Id_{V^*}) \circ b_{V} \in K.$$ 
In particular, define $\dim_{\cat}:\ob \rightarrow K$ by
\[
\dim_\cat(V)=\tr_{\cat}(\Id_V).
\]
  We call $\dim_\cat(V)$ the \emph{categorical dimension}
of $V$.

\subsection{} 
In Lemma~\ref{P:ambdim}\eqref{LI:amb2} it is assumed that the categorical
trace vanishes on the radical of $\End_{\cat}(J)$ for an absolutely
indecomposable object.  The following result yields a general scenario where
this occurs.  It is presumably well known to experts and we wrote down a proof based on one given by Deligne \cite[Lemma 3.5]{De} in the symmetric setting.
\begin{lemma}\label{L:deligneprop} 
  Let $\cat$ be a ribbon category which is an abelian category.
Let
\begin{equation}\label{E:SES}
\begin{CD}
0 @>>> A' @>r>> A @>s>> A'' @>>> 0 \\
@.     @VVf'V    @VVfV     @VVf''V    @. \\
0 @>>> A' @>r>> A @>s>> A'' @>>> 0 
\end{CD}
\end{equation}
be a morphism of short exact sequences in $\cat$.  Then 
\[
\Tr_{\cat } (f) = \Tr_{\cat }(f') + \Tr_{\cat }(f'').
\]  
\end{lemma}


\begin{proof}  Before proving the proposition, we first set the groundwork.  Recall that for $M,N$ in $\cat$ one has canonical isomorphisms of $K$-modules:
\[
\Hom_{\cat}(M,N) \cong \Hom_{\cat}(\unit , N \otimes M^{*}) \cong \Hom_{\cat}(N^{*}, M^{*}).
\]  Given $g \in \Hom_{\cat}(M,N)$ we write  $\hat{g}$ for the corresponding element in $\Hom_{\cat}(\unit, N \otimes M^{*})$ and $g^{*}$ for the corresponding element in $\Hom_{\cat}(N^{*}, M^{*})$.   To avoid confusion, given $g \in \Hom_{\cat}(M,M)$ we will write $\tr_{M}(g)$ for $\tr_{\cat}(g)$ and, in a slight abuse of notation, we write $\tr_{M}(\hat{g})$ using the above isomorphism.  Finally, we note that for such a morphism, $d'_{M} \circ \hat{g} = \tr_{M}(g)$.

Consider the short exact sequence 
\begin{equation}\label{E:SES2}
0 \to  A' \xrightarrow{r} A \xrightarrow{s} A'' \to 0.
\end{equation}
Applying the duality functor yields another short exact sequence and, recalling from Section~\ref{SS:ribboncats} that tensor functor is exact, we can tensor these together to obtain the following bicomplex which is exact everywhere:
\begin{equation}\label{E:exactsquare}
\begin{CD}
@.       0 @. 0 @. 0 @.  \\
@.     @VVV    @VVV     @VVV    @. \\
0 @>>> A' \otimes (A'')^{*} @>r \otimes 1>> A \otimes (A'')^{*} @>s \otimes 1 >> A'' \otimes (A'')^{*} @>>> 0 \\
@.     @VV1 \otimes s^{*}V    @VV1 \otimes s^{*}V     @VV1 \otimes s^{*}V    @. \\
0 @>>> A' \otimes A^{*} @>r \otimes 1>> A \otimes A^{*} @>s \otimes 1 >> A'' \otimes A^{*} @>>> 0 \\
@.     @VV1 \otimes r^{*}V    @VV1 \otimes r^{*}V     @VV1 \otimes r^{*}V    @. \\
0 @>>> A' \otimes (A')^{*} @>r \otimes 1>> A \otimes (A')^{*} @>s\otimes 1>> A'' \otimes (A')^{*}  @>>> 0 \\
@.     @VVV    @VVV     @VVV    @. \\
@.       0 @. 0 @. 0 @.  
\end{CD}
\end{equation}  Let 
\[
s \otimes r^{*} = s\otimes 1 \circ 1 \otimes r^{*}: A \otimes A^{*} \to A'' \otimes \left(A' \right)^{*}
\] and let $F= \operatorname{Ker}(s \otimes r^{*})$.   It is an easy exercise to verify that for $f$ in \eqref{E:SES} one has $ s\otimes r^{*} \circ \hat{f} = 0$; that is, the image of $\hat{f}$ lies in $F$.

A diagram chase using this diagram verifies that 
\begin{equation}\label{E:ontoF}
F = \operatorname{Im} \left(r \otimes 1 \right) + \operatorname{Im} \left(1 \otimes s^{*} \right).
\end{equation}
Namely, let $x \in F$.  Then $s \otimes 1 \left(1 \otimes r^{*} (x) \right)=0$ and, hence, $1 \otimes r^{*}(x)$ lies in the image of $r \otimes 1$.  Fix $w_{1} \in A' \otimes (A')^{*}$ so that $r \otimes 1 (w_{1}) = 1 \otimes r^{*} (x)$.  Since $1 \otimes r^{*}$ is surjective, we can choose $w_{2} \in A' \otimes A^{*}$ so that $1 \otimes r^{*}(w_{2})=w_{1}$.  Let $x'=r\otimes 1 (w_{2})$.  Now consider $x-x_{1} \in A \otimes A^{*}$.  We then have 
\begin{align*}
1\otimes r^{*}(x-x') &= 1 \otimes r^{*}(x) - 1 \otimes r^{*}(x') \\
                        & = 1 \otimes r^{*}(x) - 1 \otimes r^{*}(r\otimes 1 (w_{2})) \\
                        & = 1 \otimes r^{*}(x) - r\otimes 1 (1 \otimes r^{*}(w_{2})) \\
                        & = 1 \otimes r^{*}(x) - r\otimes 1 (w_{1}) =0.
\end{align*} Hence $x-x'$ lies in the kernel of $1 \otimes r^{*}$ and, hence, the image of $1 \otimes s^{*}$.  Therefore, since $x'$ was in the image of $r \otimes 1$, it follows that $x$ lies in the sum of the images of the morphisms $r \otimes 1$ and $1 \otimes s^{*}$, just as claimed.

We now note that we have morphisms 
\begin{align*}
\varphi'&: F \to A' \otimes \left(A' \right)^{*} \\
\varphi''&: F \to A'' \otimes \left(A'' \right)^{*}
\end{align*} defined as follows.

Let $x \in F$.  Then $0 = s \otimes r^{*}(x) = s \otimes 1 \left(1\otimes r^{*}(x) \right)$. From the bottom row of \eqref{E:exactsquare} the fact that $ s \otimes 1 \left(1\otimes r^{*}(x) \right) = 0$ implies that $1 \otimes r^{*}(x)$ lies in the image of the injective morphism $r \otimes 1$.  We can then define $\varphi' = (r \otimes 1)^{-1} \circ 1 \otimes r^{*}$. 

Similarly, $0 = s \otimes r^{*}(x) = 1 \otimes r^{*} \left(s\otimes 1(x) \right)$.  Using the rightmost column of \eqref{E:exactsquare} we see that $s\otimes 1(x)$ lies in the image of the injective morphism $ 1 \otimes s^{*}$.  We can then define $\varphi'' = (1 \otimes s^{*})^{-1} \circ s \otimes 1$. 

Since the image of $\hat{f}$ lies in $F$ we can form the morphisms $\varphi' \circ \hat{f}: \unit \to A' \otimes (A')^{*}$ and $\varphi'' \circ \hat{f}: \unit \to A'' \otimes (A'')^{*}$. We now consider  the morphism
\[
\Gamma=d'_{A'} \circ \varphi' \circ \hat{f}  + d'_{A''}\circ  \varphi'' \circ \hat{f} : \unit \to \unit 
\]  Using the commutivity of \eqref{E:SES} and the diagram calculus discussed in Section~\ref{SS:diagrams} (or direct calculation), one has that $ \varphi' \circ \hat{f} = \hat{f'}$ and $ \varphi'' \circ \hat{f} = \hat{f''}$.  Therefore we have 
\begin{equation}\label{E:onehalf}
\Gamma = \tr_{A'}(f')+\tr_{A''}(f'').
\end{equation}

On the other hand, if $u \in \unit$, then by \eqref{E:ontoF} we have that $\hat{f}(u) = r \otimes 1 (v) + 1 \otimes s^{*}(w)$ for some $v \in A' \otimes A^{*}$ and $w \in A \otimes (A'')^{*}$.  Therefore we have 
\begin{align*}
\Gamma (u)          &= (d_{A'} \circ \varphi' \circ r \otimes 1)(v) + (d_{A''}\circ \varphi'' \circ 1 \otimes s^{*})(w) \\
                    &= (d_{A'} \circ 1 \otimes r^{*})(v) + (d_{A''} \circ s \otimes 1)(w) \\
                    &= (d_{A} \circ r \otimes 1)(v) +   (d_{A} \circ 1 \otimes s^{*})(w) \\
                    &= d_{A} \circ \left( r\otimes 1 (v) + 1 \otimes s^{*}(w) \right) \\
                    &= (d_{A} \circ \hat{f})(u)\\
                    &= \tr_{A}(f)(u)               
\end{align*}
Where the first equality holds because when you expand out the expression, two terms are zero; the second equality follows from an elementary simplification; and the third equality follows from diagram calculus which shows that $d_{A'} \circ 1 \otimes r^{*} = d_{A} \circ r \otimes 1$ and $d_{A''} \circ s \otimes 1 = d_{A} \circ 1 \otimes s^{*}$.  Thus we have 
\begin{equation}\label{E:otherhalf}
\Gamma  = \tr_{A}(f).
\end{equation}  Equating \eqref{E:onehalf} and \eqref{E:otherhalf} yields the first statement of the proposition.

\end{proof}

We then have the following corollary.

\begin{corollary}\label{C:deligne}  If $\cat$ is as in the previous lemma and $J$ in $\cat$ is an absolutely indecomposable object, then $\tr_{\cat}$ is identically zero on $\Rad \left(\End_{\cat}(J) \right)$.
\end{corollary}

\begin{proof} If $f \in \Rad \left(\End_{\cat}(J) \right)$, then $f$ is a nilpotent morphism (cf.\ our assumption in Section~\ref{SS:groundring}). Say $f^{n} = 0$.  Now consider the morphism of short exact sequences in $\cat$ given by 
\begin{equation*}
\begin{CD}
0 @>>> \operatorname{Ker}(f) @>r>> A @>s>> A/\operatorname{Ker}(f) @>>> 0 \\
@.     @VV0V    @VVfV     @VV\bar{f}V    @. \\
0 @>>> \operatorname{Ker}(f) @>r>> A @>s>> A/\operatorname{Ker}(f) @>>> 0, 
\end{CD}
\end{equation*}
where $r$ and $s$ are the canonical morphisms, and $\bar{f}$ is the morphism induced by $f$.   A straightforward induction on $n$ using the above lemma proves that $\tr_{\cat}(f)=0$.

\end{proof}

\subsection{The Diagrammatic Calculus}\label{SS:diagrams}Next we will discuss how one can represent morphisms in the category $\cat$ with graphs.  The algebraic identities of the ribbon category give a graphical calculus in which graphs representing morphisms in $\cat$ can easily be minipulated.  We will not present the complete calculus here; however, we will provide the important relations which are required for the proofs.   For more details on this graphical calculus see \cite[Chapter XIV]{Kas}.

We represent a morphism $f:U \rightarrow V$ by a box with two vertical arrows as in Figure~\ref{SF:repf},  where $U,V$ are the colors of the arrows and $f$ is the color of the box.  Note that we follow the convention that the graphical depiction of morphisms should be read from bottom to top.  In the special case of $\Id_{V}: V \to V$ we commonly omit the box labeled by $\Id_{V}$ and simply draw a plain vertical line colored by $V$.  The composition of two morphisms is obtained by putting one box on top of the other. Also note that here and elsewhere we use the symbol $\dot{=}$ between two graphs to mean that the corresponding morphisms in $\cat$ are equal.  
    \begin{figure}[h]
  \centering
  \subfloat[Graph for $f$]{\label{SF:repf} \hspace{16pt} 
    $\xymatrix{ \ar[d]^{V}\\ *+[F]\txt{  $f$ } \ar[d]^{U}\\   \: }$\hspace{16pt} }                
  \subfloat[Graph for $f\circ g$]{\label{SF:fg1}\hspace{10pt} \xymatrix{ 
    \ar[d]^{W}\\
    *+[F]\txt{  $g\circ f$ } \ar[d]^{U}\\
    \: } \put(7,-37) {$\dot{=}$}\hspace{4ex} 
   $\put(-3,-74){ \put(6,24){{\small $f$}} \put(7,44){{\small $g$}} \epsfig{figure=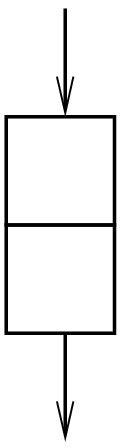,height=14ex}}$
    \hspace{20pt} 
    }
     \subfloat[Graph for $f\otimes g$]{\label{SF:fg2}
     \hspace{10pt}
     \xymatrix{  \ar[d]^{V}\\
    *+[F]\txt{  $f$ } \ar[d]^{U}\\
    \: }\xymatrix{ 
    \ar[d]^{W}\\
    *+[F]\txt{\;$\put(2,-6) \textit{g} $\; } \ar[d]^{V}\\
    \: }\put(0,-37){$\dot{=}$}\hspace{2ex}\xymatrix{ 
    \ar@< 8pt>[d]^{W}
    \ar@< -8pt>[d]_{V}\\
    *+[F]\txt{  $f \otimes g$ } \ar@< 8pt>[d]^{V}
    \ar@< -8pt>[d]_{U}\\
    \: }\hspace{10pt}
}
    \subfloat[Graph of $f$]{\label{SF:repf2}
   \hspace{10pt} \xymatrix{
  \ar@< 8pt>[d]^{W_m}_{... \hspace{1pt}}
  \ar@< -8pt>[d]_{W_1}\\
  *+[F]\txt{ \: \; f \; \;} \ar@< 8pt>[d]^{V_n}_{... \hspace{1pt}}
  \ar@< -8pt>[d]_{V_1}\\
  \: }
    \hspace{10pt}}
  \caption{}
\end{figure}

  The tensor product of two morphisms is represented by setting
the two corresponding graphs next to each other.  For example, if $f:U
\rightarrow V$ and $g:V \rightarrow W$ are morphism of $\cat$ then we
represent $g\circ f$ and $f\otimes g$ by Figures \ref{SF:fg1} and \ref{SF:fg2}, respectively.  
In general, a morphism
$f:V_1\otimes \dotsb \otimes V_n \rightarrow W_1\otimes \dotsb \otimes W_m$ in
$\cat$ can be represented by the box and
arrows given in Figure \ref{SF:repf2}.

The braiding $c_{V,W}$ and its inverse $c_{V,W}^{-1}$ are represented by Figure~\ref{F:repbraiding}.
\begin{figure}[h]
 \centering
 $\put(-7,-12){{\small $V$}} \put(33,-12){{\small $W$}} \put(4,-28){{\large  $c_{V,W}$}}  \epsh{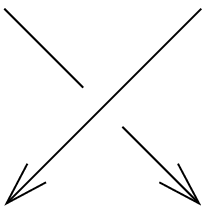}{6ex}  \hspace{60pt} {\put(-8,-12){{\small $W$}} \put(33,-12){{\small $V$}} \put(4,-28){{\large  $c_{V,W}^{-1}$}} \epsh{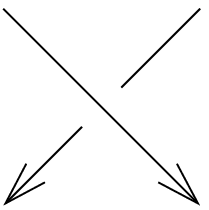}{6ex}}$
  \caption{}
  \label{F:repbraiding}
\end{figure}
The invertibility and naturality of  $c_{V,W}$ are expressed in Figures \ref{F:invbraiding} and \ref{F:natbraiding}, respectively.  The naturality of $c_{V,W}^{-1}$ gives a similar expression.
\begin{figure}[h]
 \centering
 \subfloat[Invertibility of the Braiding]{\label{F:invbraiding}
\hspace{10pt} 
$ {\put(-6,-23){{\small $V$}} \put(37,-23){{\small $W$}}\put(68,-23){{\small $V$}} \put(112,-23){{\small $W$}} \put(142,-23){{\small $V$}} \put(183,-23){{\small $W$}}
\put(53,0) {\large$\dot{=}$} \put(123,0){\large$\dot{=}$}\epsh{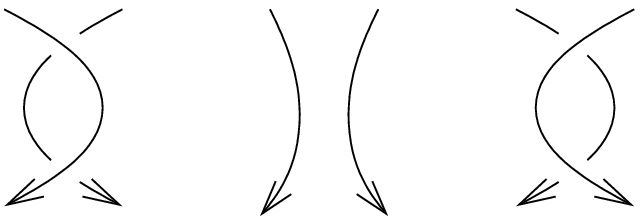}{12ex}}
$
\hspace{10pt} 
}
 \subfloat[Naturality of the Braiding]{  \label{F:natbraiding}
 \hspace{10pt} 
$\put(5,0){{ $g$}} \put(37,0){{ $f$}} \put(95,0){{ $f$}} \put(130,0){{ $g$}} 
\put(-5,40){{\small $W'$}}  \put(50,40){{\small $V'$}}    \put(88,40){{\small $W'$}}  \put(141,40){{\small $V'$}} 
 \put(4,-40){{\small $V$}}  \put(49,-40){{\small $W$}}    \put(91,-40){{\small $V$}}  \put(145,-40){{\small $W$}}   \put(69,0){\large$\dot{=}$}\epsh{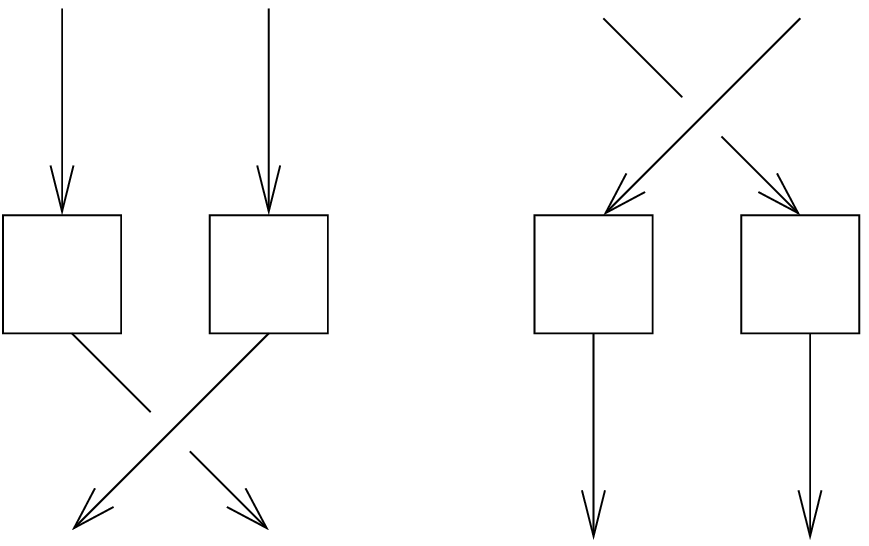}{18ex}
$
\hspace{10pt} 
}
\caption{}
\label{F:braidings}
\end{figure} 
Next we assign graphs to the duality morphisms.  
The morphisms 
$b_{V},  d_{V},  b'_{V}$ and $d'_{V}$
  are represented by the graphs in Figure \ref{F:duality}.
\begin{figure}[h]
  \begin{center}
  $  \put(17,-20){{ $b_V$}}   \put(77,-20){{ $d_V$}}   \put(145,-20){{ $b'_V$}}   \put(223,-20){{ $d'_V$}} \vcenter{\xymatrix@R=10pt @C=4pt{\ar@{-}
    `d^r[]^<<{V} `r^u[rr] |--{\SelectTips{cm}{}\object@{>}} [rr] & &\\ & &}} 
    \hspace{30pt} 
        \vcenter{\xymatrix@R=10pt @C=4pt{& &\\& &\ar@{-} `u^l[]^<<{V} `l^d[ll] |{\SelectTips{cm}{}\object@{<}} [ll] }}
        \hspace{30pt}  
     \vcenter{\xymatrix@R=10pt @C=4pt{\ar@{-} `d^r[] `r^u[rr] |{\SelectTips{cm}{}\object@{<}} [rr]^>>{V} & &\\ & &}} \hspace{30pt} 
     \vcenter{\xymatrix@R=10pt @C=4pt{& &\\& &\ar@{-} `u^l[] `l^d[ll] |--{\SelectTips{cm}{}\object@{>}} [ll]^>>{V} }}$
   \caption{}
    \label{F:duality}
  \end{center}
\end{figure}
The relations given in \eqref{E:defduality} are represented by the graphical expressions in Figure \ref{F:defduality}.
\begin{figure}[h]
\centering
$\put(12,-17){{\small $V$}}  \put(52,-17){{\small $V$}}  \put(31,0){$\dot{=}$} \put(116,17){{\small $V$}}  \put(161,17){{\small $V$}}  \put(136,0){$\dot{=}$}  \epsh{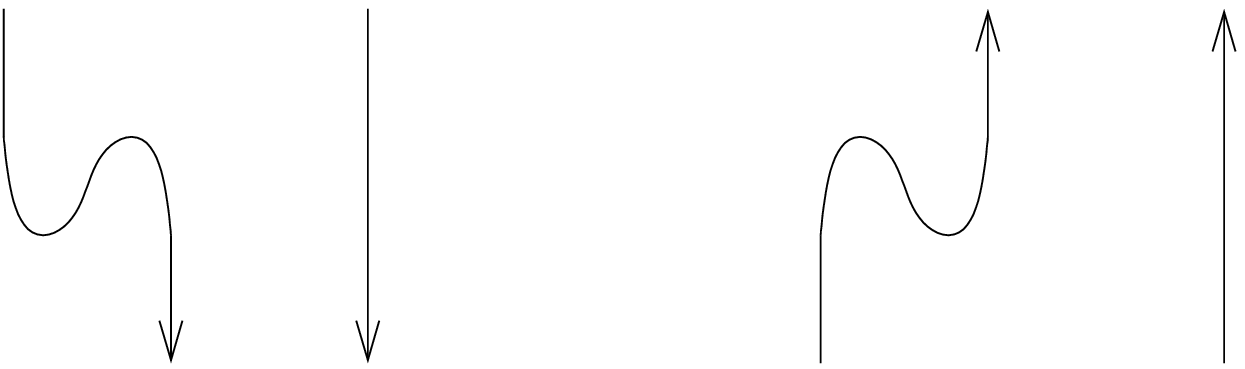}{9ex}$ 
\caption{}
\label{F:defduality}
\end{figure}

To illustrate how one uses these graphs and the graphical calculus, in Figure \ref{F:twistcal} we compute the graph corresponding to the twist.
\begin{figure}
\centering
 $ \put(31,0){ \large ${ \dot{=}}$} \put(83,-3){{ $\theta_V$}}   \put(136,0){\large$\dot{=}$}
\put(186,26){{ $\theta_V$}} \put(224,0){\large$\dot{=}$}\put(250,26){{ $\theta_V$}} \put(305,0){\large$\dot{=}$} \put(330,26){{ $\theta_V$}}\epsh{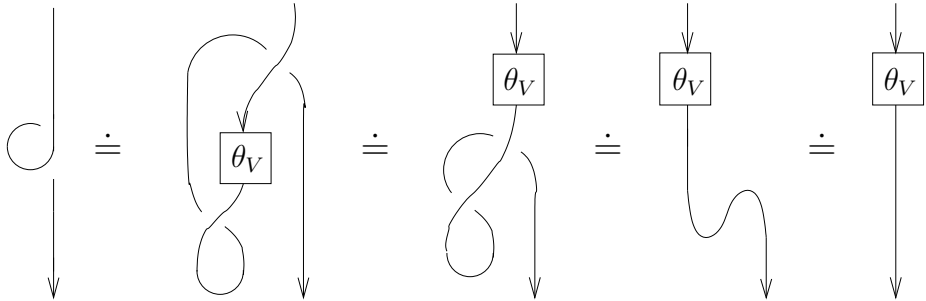}{22ex}
$
\caption{All of the edges are colored with $V$.}
\label{F:twistcal}
\end{figure}
The first equality in Figure \ref{F:twistcal} is by definition, the second and third by naturality of the braiding, the last one by Figure \ref{F:defduality}.  Similarly, one can show that the equalities represented in Figures \ref{F:twist2} and \ref{F:twistinv} hold.    In particular, note that this example shows if the twist is nontrivial then a kink can not be undone by the graphical calculus.   However, the invertibility of  $\theta_V$ implies that the morphisms represented in Figure \ref{F:twistinv2} are equal.
\begin{figure}
\centering
 \subfloat[]{\label{F:twist2}
 \hspace{10pt} 
$\put(32,0){{$\dot{=}$}}
 \put(51,-2){{\large $\theta_V$}} 
\epsh{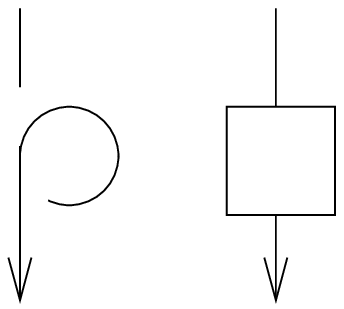}{12ex}
$
\hspace{10pt} 
}
 \subfloat[]{\label{F:twistinv}
 \hspace{10pt} 
$\put(29,0){{$\dot{=}$}}
 \put(46,-2){{\large $\theta_V^{-1}$}}
 \put(74,0){{$\dot{=}$}} 
\epsh{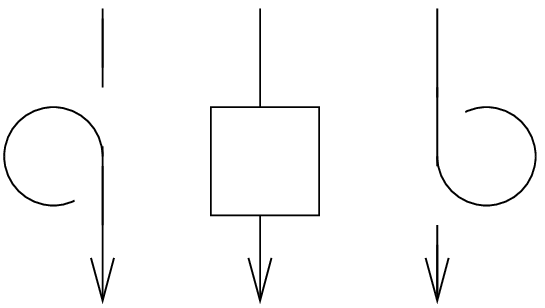}{12ex}
$
\hspace{10pt} 
}
 \subfloat[]{\label{F:twistinv2}
 \hspace{10pt} 
$\put(23,0){{$\dot{=}$}}
 \put(46,0){{$\dot{=}$}} 
\epsh{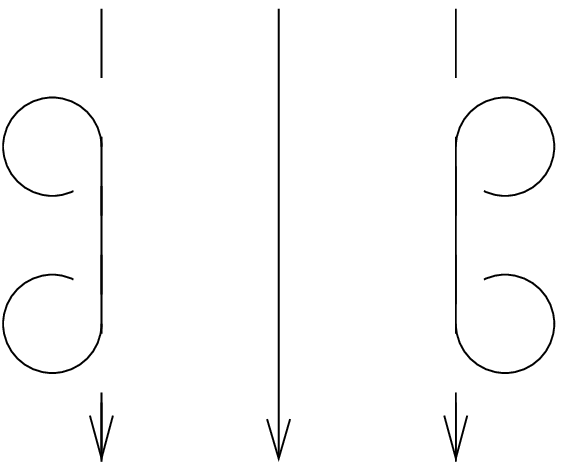}{12ex}
$
\hspace{10pt} 
}
\caption{All of the edges are colored with $V$.}
\label{}
\end{figure}

The graphs above can be described in the language of ribbon graphs and their diagrams (see \cite{Kas,Tu}).  As Figures \ref{F:braidings} and \ref{F:defduality} suggest one can consider these graphs up to isotopy.  This is similar to isotopy of framed knots or links in $\R^3$ where the plane determined by a box should always be parallel to the plane $\R\times \R\times 0$ and the line determined by the base of a box should always be parallel to the line $\R\times 0\times 0$.  Here ``framed'' refers to the fact that kinks can not be undone.    

As an exercise in the graphical calculus we give an equality which will be used later.  For morphisms $f:V\rightarrow W, g:W\rightarrow U, h:U\rightarrow V$ we have:
\begin{equation}
\label{E:partracefggf}
\tr_R\left( (\Id_V\otimes f)\circ c_{V,V}^{-1} \circ (\Id_V\otimes hg\right)=\tr_R\left( (\Id_V\otimes gf) \circ c_{V,V}^{-1} \circ (\Id_V\otimes h)\right),
\end{equation} where $\tr_{R}$ is defined in \eqref{E:trR}.
 Figure~\ref{F:partracefggf} is a graphical representation of Equation~\ref{E:partracefggf}.   The proof is provided in Figure~\ref{F:partracefggfproof} using the graphical calculus.  Specifically, we use the naturality and invertibility of the braiding, the naturality of the twist ($g \circ \theta_W=\theta_U \circ g$), and the definition of $d'_W$ to prove of \eqref{E:partracefggf}.  The left (resp. right) side of Figure~\ref{F:partracefggf} represents the same morphism as the left (resp. right) side of Figure~\ref{F:partracefggfproof}.
\begin{figure}
\centering
$\put(-5,48){{\small $V$}} \put(87,48){{\small $V$}}
\put(26,36){{\large $f$}}
\put(26,-16){{\large $h$}}
\put(26,-36){{\large $g$}}
\put(119,37){{\large $g$}}
\put(119,14){{\large $f$}}
\put(119,-38){{\large $h$}}
\put(-7,-48){{\small $V$}} \put(86,-48){{\small $V$}}
\put(53,-48){{\small $W$}} \put(147,-48){{\small $U$}}
\put(69,0) {\large$\dot{=}$} \epsh{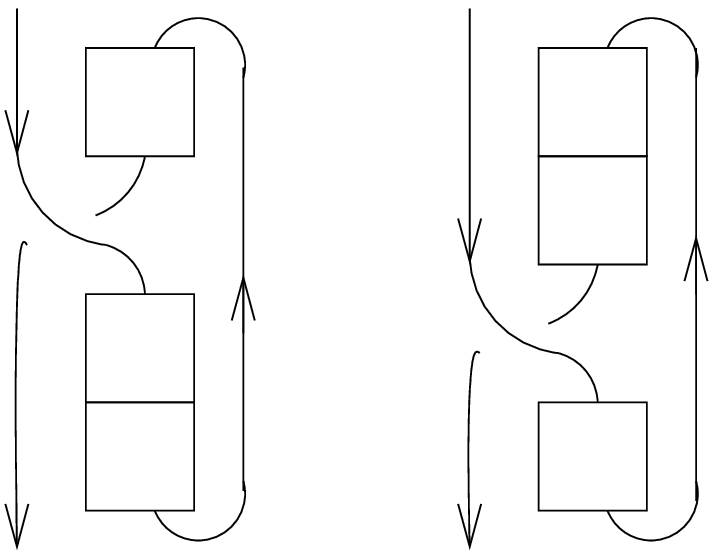}{22ex}
$
\caption{}
\label{F:partracefggf}
\end{figure}
\begin{figure}
\centering
$
\put(53,40){{\small $W$}}
 \put(-7,70){{\small $V$}} \put(-7,-70){{\small $V$}}
 \put(19,-57){{\small $U$}}
  \put(20,43){{\large $\theta_W$}}
 \put(25,23){{\large $f$}}
\put(25,-26){{\large $h$}}
\put(58,-65){{\large $g$}}
\put(183,19){{\small $W$}}
 \put(115,70){{\small $V$}} \put(115,-70){{\small $V$}}
 \put(169,-70){{\small $U$}}
  \put(141,-4){{\large $\theta_W$}}
 \put(146,-24){{\large $f$}}
\put(146,-72){{\large $h$}}
\put(165,65){{\large $g$}}
\put(296,-50){{\small $U$}}
 \put(239,70){{\small $V$}} \put(239,-70){{\small $V$}}
  \put(268,32){{\large $\theta_U$}}
  \put(272,12){{\large $g$}}
 \put(270,-9){{\large $f$}}
\put(270,-57){{\large $h$}}
\put(90,0) {\large$\dot{=}$} \put(215,0) {\large$\dot{=}$} 
\epsh{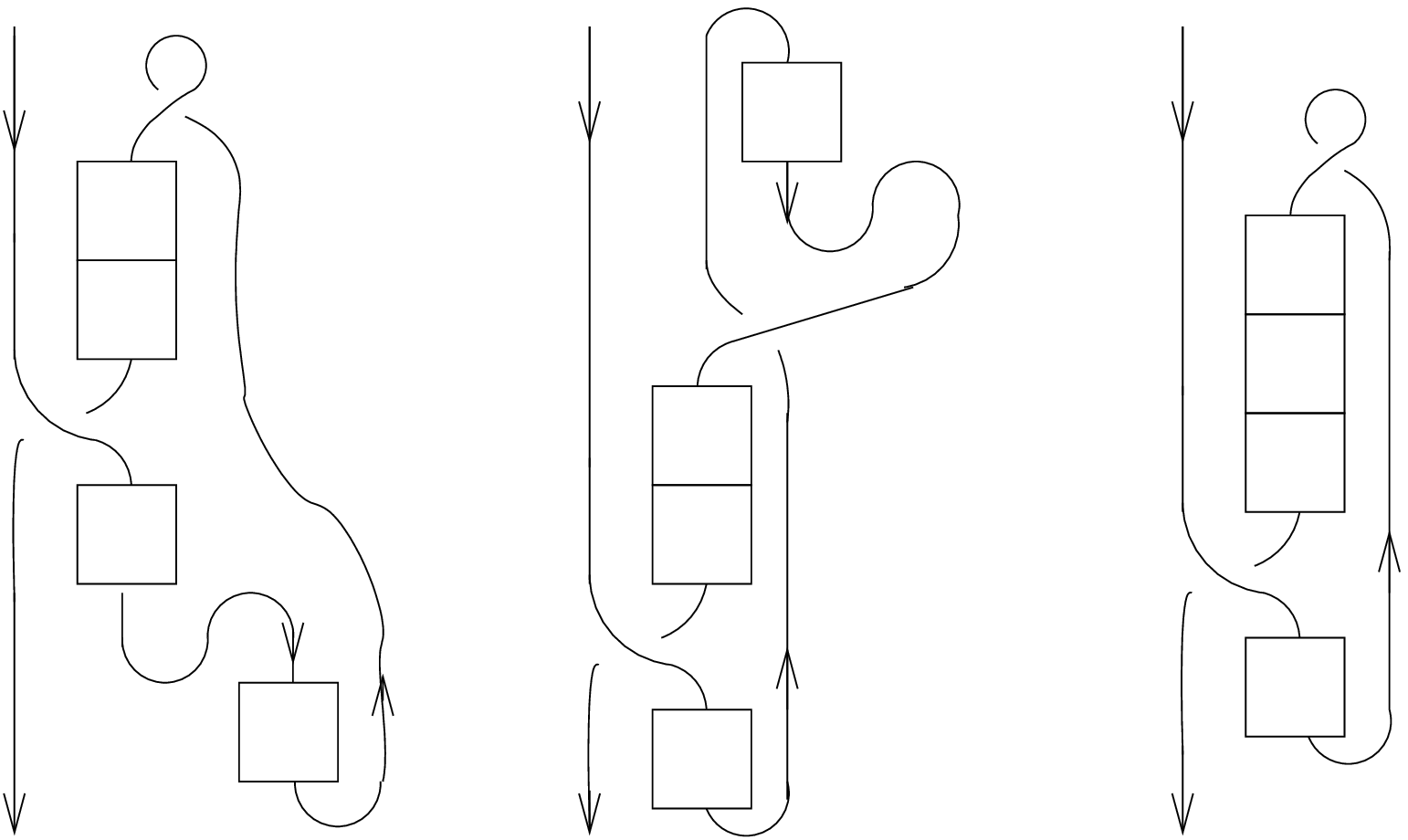}{34ex}
$
\caption{}
\label{F:partracefggfproof}
\end{figure}

\section{Generalized traces}\label{S:generalizedtraces} In this section we introduce the fundamental definitions of the paper.  The rough idea is to generalize the notion of the categorical trace by defining a generalized trace to be a family of linear functions on a full subcategory which have the desired properties.  As we will see, the key idea is the notion of an ambidextrous trace.
\subsection{Ideals}\label{SS:ideals}  We first introduce the notion of an ideal in a ribbon Ab-category.
\begin{definition}\label{D:ideal} We say a full subcategory $\ideal$ of a ribbon Ab-category $\cat$ is an \emph{ideal} if the following two conditions are met:
\begin{enumerate}
\item  If $V$ is an object of $\ideal$ and $W$ is any object of $\cat$, then $V\otimes W$ is an object of $\ideal$.
\item  $\ideal$ is closed under retracts; that is, if $V$ is an object of $\ideal$, $W$ is any object of $\cat$, and if there exists morphisms $f:W \rightarrow V$, $g:V\rightarrow W$ such that $g \circ f=\Id_W$, then $W$ is an object of $\ideal$.
\end{enumerate}

\end{definition}

Let us point out that a ribbon Ab-category is not neccessarily additive and, hence, we do not require that an ideal be closed under direct sums as the reader might expect.

The main example of an ideal which is used in this paper is constructed as follows.  Let $V$ be a fixed object of $\cat$. Let $\ideal_V$ be the full subcategory of all objects $U$ satisfying  the property that there exists an object $W$ and morphisms $\alpha:U\rightarrow V\otimes W$ and $\beta:V\otimes W \rightarrow U$ with $\beta\circ \alpha=\Id_U$.  It is not difficult to verify $\ideal_{V}$ forms an ideal.  

The following lemma records some basic properties of $\ideal_{V}$.  Note that in the proof and thereafter,  for brevity we sometimes write $fg$ to denote $f \circ g$ for morphisms $f$ and $g$ in $\cat$.

\begin{lemma}\label{L:Ideal} Let $U,V\in \ob$. Then the following statements are true.
\begin{enumerate}
\item \label{LI:UV} If $U\in \ideal_V$ then $\ideal_U\subset \ideal_V$.
\item \label{LI:VV*} $\ideal_V=\ideal_{V^*}$.
\end{enumerate}
\end{lemma}
\begin{proof}
Since $U\in \ideal_V$ there exists $W$, $\alpha:U\rightarrow V\otimes W$ and $\beta:V\otimes W \rightarrow U$ with $\beta\circ \alpha=\Id_U$.  Similarly, if $U'\in \ideal_U$ then  there exists $W'$, $\alpha':U'\rightarrow U\otimes W'$ and $\beta':U\otimes W' \rightarrow U'$ with $\beta'\circ \alpha'=\Id_{U'}$.  Now let $\alpha'':U'\rightarrow V\otimes W \otimes W'$ and $\beta'':V\otimes W \otimes W' \rightarrow U'$ be the morphisms given by $\alpha''=(\alpha\otimes \Id_{W'})\alpha'$ and $\beta''=\beta'(\beta\otimes \Id_{W'})$.  Then $\beta''\circ \alpha''=\Id_{U'}$ and so $U'\in\ideal_V$.  Thus, $\ideal_U\subset \ideal_V$ and we have proved item \eqref{LI:UV} of the lemma.

To prove item \eqref{LI:VV*} let $\alpha: V\rightarrow V^*\otimes V\otimes V$ and $\beta:  V^*\otimes V\otimes V \rightarrow V$ be the morphisms given by $\alpha=(c_{V,V^*}\otimes \Id_V)(\Id_V\otimes b_V')$ and $\beta=(d_V'\otimes \Id_V)(c_{V,V^*}^{-1}\otimes \Id_V)$ then $\beta \circ \alpha=\Id_V$.  So $V\in \ideal_{V^*}$ and item \eqref{LI:UV} of the lemma implies that $\ideal_V\subset \ideal_{V^*}$.  Similarly $\ideal_{V^*}\subset \ideal_V$.
\end{proof}

\subsection{Traces}\label{SS:traces}  We can now give the fundamental definitions of the paper.

First, for any objects $V,W$ of $\cat$ and any endomorphism $f$ of $V\otimes
W$, set
\begin{equation}\label{E:trL}
\tr_{L}(f)=(d_{V}\otimes \Id_{W})\circ(\Id_{V^{*}}\otimes
f)\circ(b'_{V}\otimes \Id_{W}) \in \End_{\cat}(W),
\end{equation} and
\begin{equation}\label{E:trR}
\tr_{R}(f)=(\Id_{V}\otimes d'_{W}) \circ (f \otimes \Id_{W^{*}})
\circ(\Id_{V}\otimes b_{W}) \in \End_{\cat}(V).
\end{equation}

\begin{definition}\label{D:trace}  If $\ideal$ is an ideal in $\cat$ then a \emph{trace on $\ideal$} is a family of linear functions
$$\{\mt_V:\End_\cat(V)\rightarrow K\}$$
where $V$ runs over all objects of $\ideal$ and such that following two conditions hold.
\begin{enumerate}
\item  If $U\in \ideal$ and $W\in \ob$ then for any $f\in \End_\cat(U\otimes W)$ we have
\begin{equation}\label{E:VW}
\mt_{U\otimes W}\left(f \right)=\mt_U \left( \tr_R(f)\right).
\end{equation}
\item  If $U,V\in \ideal$ then for any morphisms $f:V\rightarrow U $ and $g:U\rightarrow V$  in $\cat$ we have 
\begin{equation}\label{E:fggf}
\mt_V(g\circ f)=\mt_U(f \circ g).
\end{equation} 
\end{enumerate}
\end{definition}  We remark that it
follows from \eqref{E:fggf} that a trace necessarily vanishes on commutators.

\begin{definition}\label{D:ambitrace} For $V\in \ob$ we say a linear function
  $\mt:\End_\cat(V)\rightarrow K$ is an \emph{ambidextrous trace on $V$} if for
  all $f\in \End_\cat(V\otimes V)$ we have
$$\mt(\tr_L(f))=\mt(\tr_R(f)).$$
\end{definition}  

 Recall that $V$ in $\cat$ is said to be absolutely irreducible if
 $\End_{\cat}(V)=K$, absolutely indecomposable if  $\End_{\cat}(V)/\Rad
 (\End_{\cat}(V)) \cong K$, and in either case we write $\langle \; \rangle:\End_{\cat}(V) \to K$ for the canonical linear map.  
\begin{definition}\label{D:ambidef}
We say an object $J$ of $\cat$ is \emph{ambidextrous} if $J$ is an
absolutely indecomposable object whose canonical linear map defines a non-zero ambidextrous trace on $J$.
\end{definition}

For short we say $J$ is \emph{ambi} if $J$ is ambidextrous. 
\subsection{Fundamental Properties of a Trace on an Ideal}\label{SS:fundamentals}
Our first results show that the notion of a trace on an ideal and an ambidextrous trace on an object are intimately related concepts.  The proofs are most easily expressed using the graphical calculus on $\cat$ although the more algebraically minded reader can easily translate the proofs into ones which directly use the axioms of a ribbon Ab-category.

\begin{theorem}\label{T:inducedtrace}
If $\ideal$ is an ideal of $\cat$ and $\{\mt_V\}_{V\in \ideal}$ is a trace on $\ideal$, then $\mt_V$ is an ambidextrous trace for all $V\in \ideal$.
\end{theorem}
\begin{proof}
Let $V$ be an object of $\cat$ and let $f\in\End_\cat(V\otimes V)$.  From Equation \eqref{E:VW} we have $\mt_{V\otimes V}(f)=\mt_V(\tr_R(f))$.  On the other hand, Equations \eqref{E:VW} and \eqref{E:fggf} imply that
\begin{align*}
\mt_{V\otimes V}(f)&=\mt_{V\otimes V}(c_{V,V}c_{V,V}^{-1}f)=\mt_{V\otimes V}(c_{V,V}^{-1}fc_{V,V})=\mt_{V}(\tr_R(c_{V,V}^{-1}fc_{V,V})).
\end{align*}
Therefore, it suffices to show that $\tr_R(c_{V,V}^{-1}fc_{V,V})=\tr_L(f)$.  To do this we use the graphical calculus discussed above.  In particular, we have the following:
\begin{equation}
\label{E:lrtrace}
\put(6,-1){\large{ $f$}}   \put(42,-1){ \large ${ \dot{=}}$} 
\put(69,-1){\large{ $f$}}   \put(104,-1){ \large ${ \dot{=}}$} 
\put(132,-1){\large{ $f$}}   \put(171,-1){ \large ${ \dot{=}}$} 
\put(200,-1){\large{ $f$}}   \put(235,-1){ \large ${ \dot{=}}$} 
\put(276,1){\large{ $f$}}   \put(304,-1){ \large ${ \dot{=}}$} 
\put(344,1){\large{ $f$}}  
\epsh{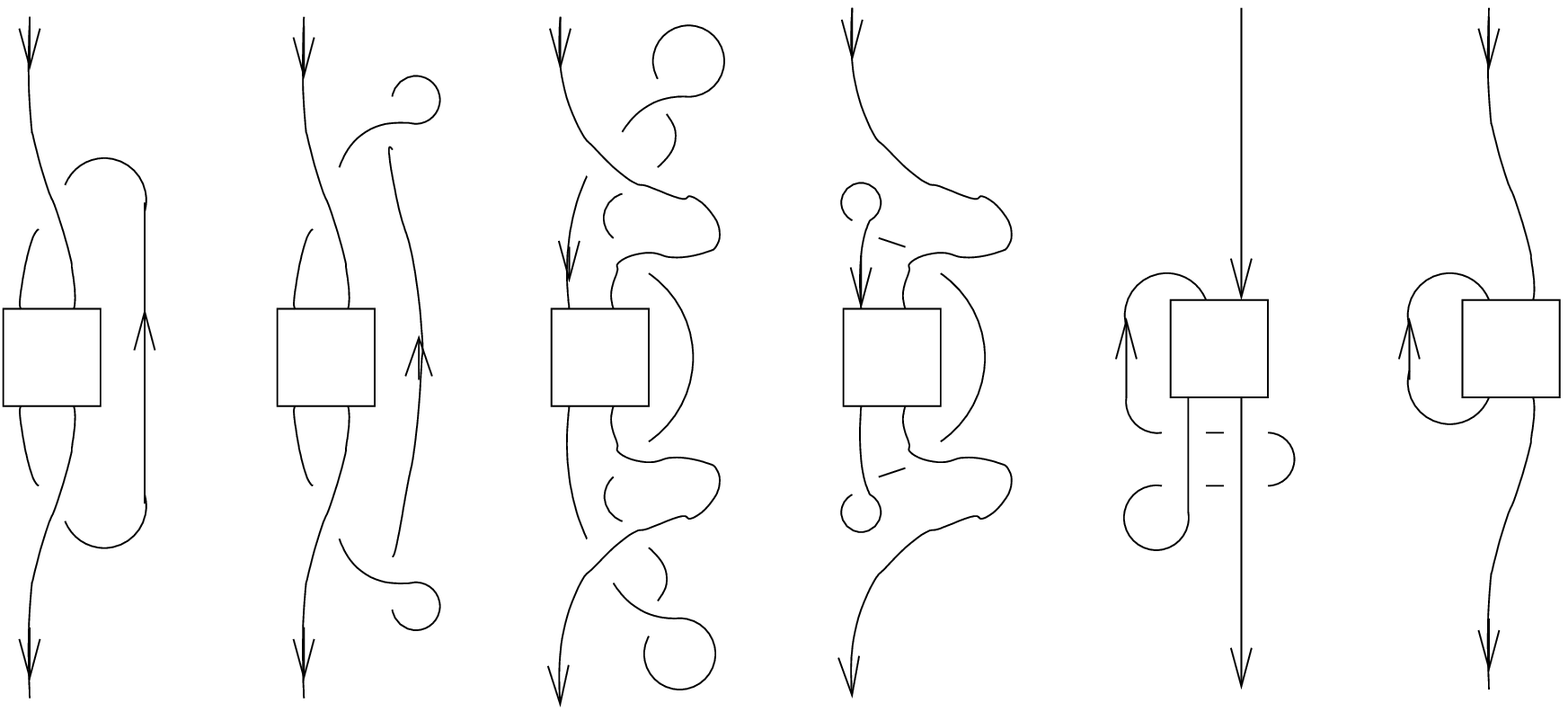}{32ex}
\end{equation}
where all arrows are colored with $V$.  The first equality of Equation \eqref{E:lrtrace} follows from the fact that $\theta_V$ is invertible (see Figure \ref{F:twistinv2}), the second and fifth from Figure \ref{F:invbraiding} and the third and fourth from the naturality of the braiding.  
\end{proof}

\begin{theorem}\label{T:ambi!Trace}
  Let $J$ be an object of $\cat$ and $\mt$ be an ambidextrous trace on $J$.
  Then there is an unique trace $\{\mt_V\}_{V \in \ideal_{J}}$ on $\ideal_J$
  with $\mt=\mt_J$.
\end{theorem}
\begin{proof}
For each $U\in \ideal_J$ choose an object $V$ in $\cat$ and morphisms $\alpha:
U \rightarrow J\otimes V$ and $\beta: J\otimes V \rightarrow U$ such that
$\beta \circ \alpha=\Id_U$.    For $f\in \End_\cat(U)$ define
$\mt_U(f)=\mt(\tr_R(\alpha \circ f \circ \beta))$
which graphically is: 
$$\mt_U(f)=\mt\left(\pic{0.7ex}{
    \begin{picture}(10,24)(0,-2)
      \qbezier(5, 2)(5, 0)(7, 0)
      \qbezier(7, 0)(9, 0)(9, 4)
      \put(2,2){\vector(0,-1){3}}
      \put(2,17){\line(0,1){4}}
      \put(9,4){\vector(0,1){12}}
      \multiput(0,2)(7,0){2}{\line(0,1){4}}
      \multiput(0,2)(0,4){2}{\line(1,0){7}}
      \multiput(0,7)(7,0){2}{\line(0,1){5}}
      \multiput(0,7)(0,5){2}{\line(1,0){7}}
      \multiput(0,13)(7,0){2}{\line(0,1){4}}
      \multiput(0,13)(0,4){2}{\line(1,0){7}}
      \put(3.5,6){\line(0,1){1}}
      \put(3.5,12){\line(0,1){1}}
      \qbezier(5, 17)(5, 20)(7, 20)
      \qbezier(7, 20)(9, 20)(9, 16)
      \put(2.5,3.1){$\beta$}
      \put(2.5,8.7){$f$}
      \put(2.5,14){$\alpha$}
    \end{picture}}\right)$$

We need to show that $\mt_U$ is independent of $\alpha$, $\beta$ and $V$.   For $U\in \ideal$ suppose that  $V'\in \ob$ and $\alpha': U \rightarrow J\otimes V'$, $\beta': J\otimes V' \rightarrow U$ are morphisms such that $\beta' \circ \alpha'=\Id_U$.   Consider the morphism $\psi:J\otimes J\rightarrow J\otimes J$ given by
$$\psi=\left(\theta_J\otimes \Id_J\otimes d'_{V}\right)\left([(\Id_J\otimes (\alpha f\beta'))(c^{-1}_{J,J}\otimes\Id_{V'})(\Id_J\otimes(\alpha'\beta))]\otimes \Id_{V^*}\right)\left(\Id_J\otimes \Id_J\otimes b_V\right)$$
where $\theta_{J}:J\rightarrow J$ is the twist of $J$.   A graphical representation of $\psi$ is given in Figure~\ref{F:psi}.   
\begin{figure}
  \centering
   \subfloat[Morphism $\psi$]{\label{F:psi}
   \hspace{10pt} 
  $  \put(38,49){{\large $\alpha$}}  \put(38,30){{\large $f$}} \put(37,10){{\large $\beta'$}} 
   \put(42,-7){ {\footnotesize $V'$}} \put(37,-32){{\large $\alpha'$}}
   \put(38,-51){{\large $\beta$}}  \put(43,-70){ {\small  $V$}} 
   \put(24,-69){ {\small  $J$}} \put(2,-69){ {\small  $J$}}   
   \put(24,69){ {\small  $J$}} \put(2,69){ {\small  $J$}}  
    \epsh{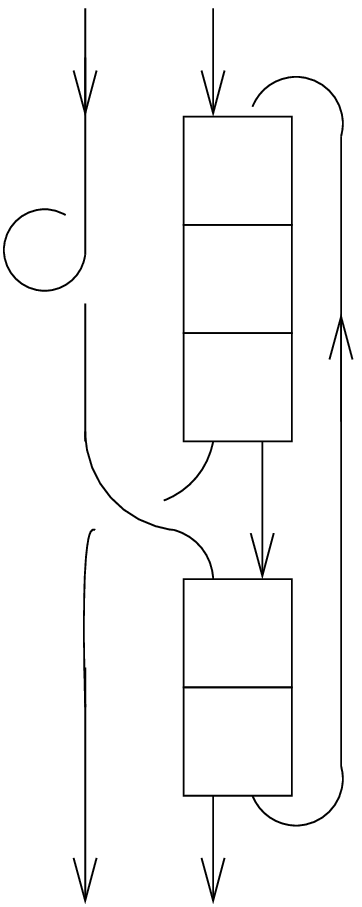}{30ex} $
    \hspace{10pt} }
 \subfloat[Morphism $\eta$]{\label{F:eta}
\hspace{10pt} 
$ 
  \put(12,43){{\large $\alpha_2$}} \put(48,43){{\large $\alpha_1$}}
   \put(13,-14){{\large $f$}} \put(49,-14){{\large $g$}}
   \put(12,-35){{\large $\beta_2$}} \put(48,-35){{\large $\beta_1$}}
    \put(0,88){ {\small  $J$}} \put(0,-88){ {\small  $J$}} 
    \put(23,88){ {\small  $J$}} \put(22,-88){ {\small  $J$}} 
  \put(42,-84){ {\small  $V_2$}} \put(53,-58){ {\small  $V_1$}} 
    \put(28,28){ {\scriptsize  $U_2$}} \put(54,20){ {\scriptsize   $U_1$}} 
\epsh{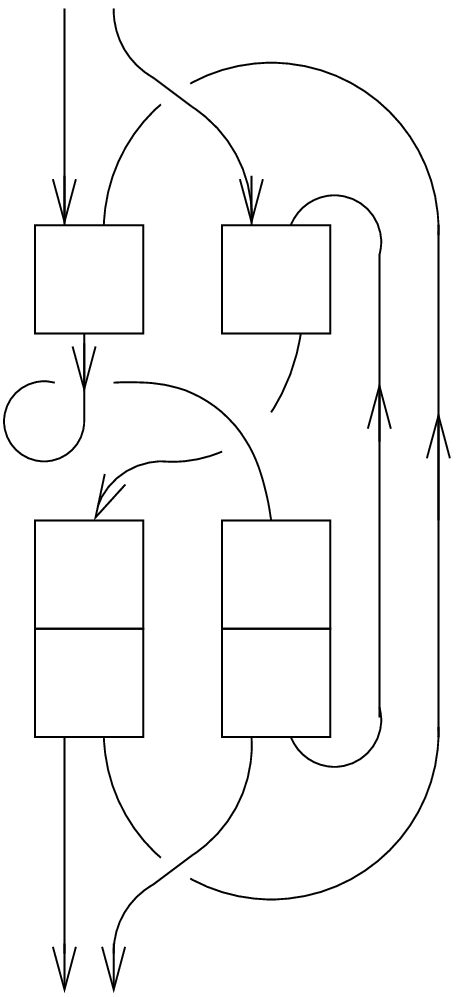}{37ex}
$
\hspace{10pt} 
}
   \caption{}
\end{figure}
The naturality of the braiding and the fact that $\beta'((\theta_J\theta_J^{-1})\otimes \Id_{V'})\alpha'=\Id_U$ we have $\tr_L(\psi)=\tr_R(\alpha\circ f\circ \beta)$.  On the other hand, the naturality of the braiding and Figure \ref{F:partracefggf} implies that  $\tr_R(\psi)=\tr_R(\alpha'\circ f\circ \beta')$.  But $\mt(\tr_L(\psi))=\mt(\tr_R(\psi))$ since $\mt$ is an ambidextrous trace on $J$.   Thus, the definition of $\mt_U$ is independent of the choices made above.

Let us prove that the family $\{\mt_U\}$ satisfies Equation \eqref{E:fggf}.  
Let $U_1,U_2\in \ideal$ and $f:U_2\rightarrow U_1 $, $g:U_1\rightarrow U_2$ be morphisms of $\cat$.   Let $\alpha_i: U_i \rightarrow J\otimes V_i$ and $\beta_i: J\otimes V_i \rightarrow U_i$ be morphisms such that $\beta_i\circ \alpha_i=\Id_{U_i}$ for $i=1,2$.    

Define the morphism  $\eta:J\otimes J \rightarrow J\otimes J$  by
\begin{align}\label{E:defeta}
\eta= (\Id_J\otimes & \Id_J \otimes d'_{V_2}) (\Id_J\otimes c^{-1}_{J,V_2}\otimes d'_{V_1}\otimes \Id_{V_2^*})\circ \notag \\
&\circ ([(\alpha_2 \otimes \alpha_1)    (\theta_{U_2}\otimes \Id_{U_1})c^{-1}_{U_2,U_1}(f\otimes g)  (\beta_2 \otimes \beta_1) ]  \otimes \Id_{V_1^*}\otimes \Id_{V_2^*})\circ \notag \\
 &\circ (\Id_J\otimes c_{J,V_2}\otimes b_{V_1} \otimes\Id_{V_2^*})(\Id_J\otimes \Id_J \otimes b_{V_2}).
\end{align}
 A graphical representation of $\eta$ is given in Figure \ref{F:eta}.   As above using the diagrammatic calculus, one can see that  $\tr_L(\eta)=\tr_R(\alpha_1\circ fg\circ \beta_1)$ and $\tr_R(\eta)=\tr_R(\alpha_2\circ gf\circ \beta_2)$.
Since $\mt$ is an ambidextrous trace on $J$ we have $\mt(\tr_L(\eta))=\mt(\tr_R(\eta))$.  Thus, the definition of the family $\{\mt_U\}$ implies
$$ \mt_{U_1}(fg)=\mt(\tr_R(\alpha_1\circ fg\circ \beta_1))=\mt(\tr_R(\alpha_2\circ gf\circ \beta_2))=\mt_{U_2}(gf).$$    

Next, we show the family $\{\mt_U\}$ satisfies Equation \eqref{E:VW}.   
Let $U\in \ideal_J$ and choose $V$, $\alpha$ and $\beta$ as above.  Let $W\in
\ob$ and $f\in \End_\cat(U\otimes W)$.  Set $V'=V\otimes W$.  Let
$\alpha':U\otimes W\rightarrow J\otimes V'$ and $\beta':J\otimes
V'\rightarrow U\otimes W$ be the morphisms given by $\alpha \otimes \Id_W$
and $\beta \otimes \Id_W$, respectively.  These morphisms satisfy
$\beta'\circ \alpha'=\Id_{U\otimes W}$.  Then
$$\tr_R(\alpha'\circ f\circ \beta')=\tr_R((\alpha \otimes d'_{W})(f\otimes \Id_{{W}^*})(\beta \otimes b_{W}))=\tr_R(\alpha\circ \tr_R(f) \circ \beta).$$
Thus, by applying $\mt$ to the last equation we see $\mt_{U\otimes W}(f)=\mt_U(\tr_R(f))$.

Finally, we show this trace is unique.   First, it is clear that $\mt_J=\mt$.  Suppose that $\{\mt'_U\}$ is a potentially different trace on $\ideal_J$ with $\mt'_J=\mt$.    For $U\in \ideal_J$ choose $V$, $\alpha$ and $\beta$ as above.  Let $f\in \End_\cat(U)$.  Then
\begin{align*}
\mt'_U(f)=\mt'_U(f\circ \beta \circ \alpha)=\mt'_{J\otimes V}(\alpha \circ f \circ \beta)=\mt'_J(\tr_R(\alpha \circ f\circ \beta))=\mt(\tr_R(\alpha \circ f\circ \beta))=\mt_U(f).
\end{align*} 
This concludes the proof of the theorem.
\end{proof}

The following result follows immediately from the previous theorem.
\begin{corollary}
  Let ${V}$ be an absolutely irreducible ambi object.  Then there is a
  unique non-zero trace on $\ideal_{V}$ up to multiplication by an element of
  $K$.
\end{corollary}

In light of the previous theorem, the existence of a non-zero trace on $\ideal_{V}$ for an object $V$ in $\cat$ amounts to verifying that $V$ admits an ambidextrous trace.  The following lemma provides several tools for doing this.

\begin{lemma}\label{P:ambdim}  Let $J$ be an object of the ribbon Ab-category $\cat$.
\begin{enumerate}
\item  If the braiding
  $c_{J,J}$ commutes with any element of $\End_\cat(J\otimes J)$,
  then any linear map on $\End_\cat(J)$ is an ambidextrous trace on $J$. \label{LI:amb1}
\item Assume $\cat$ is as in Lemma~\ref{L:deligneprop} and $J$ is an absolutely indecomposable object of $\cat$ such that the elements of $\Rad \left(\End_{\cat}(J) \right)$ are nilpotent.  If $\dim_\cat(J)\neq0$, then any scalar multiple of the canonical trace on $\End_\cat(J)$ is an ambidextrous trace on $J$. \label{LI:amb2}
\end{enumerate}
\end{lemma}
\begin{proof}
Let $f \in\End_{\cat}(J \otimes J)$.  To prove \eqref{LI:amb1} it is enough to show that $\tr_R(f)=\tr_L(f)$.  From Equation \eqref{E:lrtrace} we have
  
\begin{equation}\label{E:flipconj}
\tr_{R}(f)=\tr_{L}(c_{J,J}^{-1}\circ f \circ c_{J,J}).
\end{equation}

But $c_{J,J}$ commutes with $\End_{\cat}(J\otimes J)$ and so
$c_{J,J}^{-1}\circ f \circ c_{J,J}=f$.  

To prove \eqref{LI:amb2} we first observe that if $\langle \, \rangle$ denotes the canonical trace and $\tr_{\cat}$ the categorical trace, then both vanish on $\Rad \left(\End_{\cat}(V) \right)$; by definition for the canonical trace and by Corollary~\ref{C:deligne} for the categorical trace.  From this we obtain 
\[
\dim_{\cat}(J)\langle h \rangle = \tr_{\cat}(h)
\] for any $h \in \End_{\cat}(J)$.  Furthermore, for any $f \in  \End_{\cat}(J \otimes J)$ we have 
\[
\tr_{\cat}(\tr_{L}(f))=\tr_{\cat}(f) = \tr_{\cat}(\tr_{R}(f)).
\]  Combining these observations we have
\[
\dim_{\cat}(J)\langle \tr_{L}(f) \rangle =  \tr_{\cat}(\tr_{L}(f)) = \tr_{\cat}(f) = \tr_{\cat}(\tr_{R}(f)) = \dim_{\cat}(J)\langle \tr_{L}(f) \rangle.
\]  Therefore the canonical trace on $J$ is ambidextrous.

\end{proof}

\begin{remark}\label{R:ambi}
Let us present several several situations in which the above lemma proves useful.  
\begin{enumerate}
\item If the tensor product $J\otimes J$ is semisimple and multiplicity free then $\End_\cat(J\otimes J)$ is commutative and the lemma implies that any linear map on $\End_\cat(J)$ is an ambidextrous trace on $J$.
\item If $c_{J,J}^{2}=1$ and the characteristic of $K$ is not two, then the conjugation action of $c$ on $\End_{\cat}(J\otimes J)$ is semisimple.  The endomorphism algebra decomposes into $\pm 1$ eigenspaces under this action.  If the $-1$ eigenspace is zero, then $c$ is central and again any linear map on $\End_{\cat}(J)$ provides an ambidextrous trace. 

Conversely, if $f \in \End_{\cat}(J \otimes J)$ is in the $-1$ eigenspace, then by \eqref{E:flipconj} one has that 
\[
\tr_{R}(f) = - \tr_{L}(f).
\]  Therefore, if $\mt: \End_{\cat}(J) \to \End_{\cat}(\unit )$ is a linear map such that $\mt\left( \tr_{L}(f)\right) \neq 0$, then $\mt$ does not define an ambidextrous trace on $J$. 
\end{enumerate} 
\end{remark}

We pause to consider the following simple example.

\begin{example}\label{Ex:unit}  Let $J=\unit$ be the unit object in $\cat$.  Since $\End_{\cat}(\unit )$ is a commutative ring, Lemma~\ref{P:ambdim}\eqref{LI:amb1} implies that the identity map $\langle \; \rangle : \End_{\cat}(\unit ) \to K$ defines an ambidextrous trace on $\unit$.   By Theorem~\ref{T:ambi!Trace} we obtain a trace on $\ideal_{\unit} = \cat$ induced by $\langle \; \rangle$.  We then have 
\[
t_{V} = \tr_{\cat}
\] for all $V$ in $\cat$.  In this way we recover the categorical trace on $\cat$.  Similarly we recover the categorical dimension since 
\[
\md_{\unit} = \dim_{\cat},
\] where $\md_{\unit}$ is the modified dimension defined in the next section.

\end{example}

\section{Modified dimensions}\label{S:ModDim}
\subsection{The Modified Dimension Function}\label{SS:moddim}  
We now use the trace on an ideal $\ideal$ introduced above to define a modified
dimension function on objects in the ideal $\ideal$.  Namely, let $\ideal$ be an ideal in an Ab-ribbon category $\cat$ and let $\mt =\left\{\mt_{V} \right\}_{V\in \ideal }$ be a trace on $\ideal$.  We define the \emph{modified dimension function} 
\[
\md_{\mt}: \operatorname{Ob}(\ideal ) \to  K
\] by the formula 
\[
\md_{\mt} (V) = \mt_{V}\left(\Id_{V} \right).
\]

We will primarily be interested in the ideal $\ideal_{J}$ where $J$ is an absolutely
indecomposable object in $\cat$.
Let $J$ be absolutely indecomposable and recall that we write $\langle \; \rangle: \End_{\cat}(J) \to K$ for the canonical projection.  Recall that we assume the elements
of $\Rad \left(\End_{\cat}(J) \right)$ are nilpotent.  We remark that then for any $f \in \End_{\cat}(J)$ we have
\[
f = \langle f \rangle \Id_{J} + r,
\] 
where $r \in \Rad \left(\End_{\cat}(J) \right)$ and $f$ 
is invertible if and only if $\langle f \rangle $ is non-zero.

\begin{definition}\label{D:moddimdef}  
Fix an ambi object $J$ with canonical linear map $\langle\; \rangle$ and let $\{\mt_V\}_{V\in \ideal_J}$ be the trace on $\ideal_J$ coming from Theorem \ref{T:ambi!Trace} applied to $\langle \; \rangle$.   Define the \emph{modified dimension $\md_{J}$} to be the function from objects of $\ideal_J$ to $K$ given by 
\begin{equation}\label{E:moddimdef}
\md_J(V)=\mt_{V}(\Id_{V})=\langle \tr_R(\alpha \circ \beta) \rangle
\end{equation}
where $\alpha:V\rightarrow J\otimes W$ and $\beta: J\otimes W \rightarrow V$ are morphisms such $\beta \circ \alpha=\Id_V$ for some $W\in \ob$.
\end{definition}

\subsection{The Modified Dimension and Ideals}\label{SS:moddimandideals}  Throughout this section we assume $J$ is an ambidextrous object in $\cat$.   The following results show a close relationship between the modified dimension function and ideals.

\begin{theorem}\label{T:dim0}
Let $V$ be any object in $\ideal_J$.  If $\md_J(V)\neq 0$ then $\ideal_V=\ideal_J$.
\end{theorem}
\begin{proof}
Since $\md_J(V) \neq 0$ then $\tr_R(\alpha \circ \beta)$ is an invertible endomorphism of $J$.   Now, let $\alpha':J\rightarrow V\otimes W^*$ and $\beta': V\otimes W^* \rightarrow J$ be the morphisms given by $\alpha'= ( \beta \otimes \Id_{W^*} )\circ ( \Id_J \otimes b_W )$ and $\beta'= (\Id_J\otimes d'_W) \circ (\alpha \otimes \Id_{W^*})$.   Thus, $\beta'\circ\alpha'=\tr_R(\alpha \circ \beta)$ and since $\tr_R(\alpha \circ \beta)$ is invertible we have $J\in \ideal_V$ and $\ideal_J \subset \ideal_V$.  On the other hand, $\ideal_V \subset \ideal_J$ as $V\in \ideal_J$.
\end{proof}

\begin{lemma}\label{L:dim0}
  Let $V$ be an absolutely simple object which is an object of $\ideal_J$.  Then
  $\md_J(V)\neq 0$ if and only if $\ideal_J=\ideal_V$.
\end{lemma}
\begin{proof}
If $\md_J(V)\neq 0$ then Theorem \ref{T:dim0} implies $\ideal_J=\ideal_V$.  On the other hand, if 
$\ideal_V=\ideal_J$ then $J\in \ideal_V$ and so there exists $W$, $\alpha: J\rightarrow V\otimes W$ and $\beta: V\otimes W \rightarrow J$ such that $\beta\circ\alpha=\Id_J$.  Also, since $V\in \ideal_J$ there   exists $W'$, $\alpha': V\rightarrow J\otimes W'$ and $\beta': J\otimes W' \rightarrow V$ such that $\beta'\circ\alpha'=\Id_J$.   Consider the maps $\alpha'':J\rightarrow J\otimes (W'\otimes W)$ and  $\beta'':J\otimes (W'\otimes W)\rightarrow J$ determined by $(\alpha'\otimes \Id_W)\alpha$ and  $\beta(\beta'\otimes \Id_W)$, respectively.   Notice that $\beta''\circ \alpha''=\Id_J.$  So we can compute $\md_J(J)$ with these morphisms, in particular 
$$\md_J(J)=\mt(\Id_J)=\mt_J(\Id_J)=\langle\tr_R(\alpha''\circ \beta'')\rangle.$$
Now since $V$ is absolutely simple we have $\langle\tr_R(\alpha''\circ \beta'')\rangle=\langle\tr_R(\alpha'\circ \beta')\rangle\langle\tr_R(\alpha\circ \beta)\rangle$ which is non-zero as $\mt(\Id_J)=1$.  Thus, $\md_J(V)=\langle\tr_R(\alpha\circ \beta)\rangle$ is invertible and so non zero.
\end{proof}

\begin{corollary}\label{C:dimthen0}
  Let $V$ be an absolutely simple object in $\ideal_J$ and $U\in\ideal_V$.  If
  $\md_J(V)= 0$ then $\md_J(U)=0$.
\end{corollary}
\begin{proof}
Suppose that $\md_J(U)\neq 0$ then $\ideal_U=\ideal_J$.  Then $\ideal_V=\ideal_J$ as $\ideal_J=\ideal_U\subset \ideal_V=\ideal_V$.  Finally, since $V$ is absolutely simple then $\md_J(V)\neq 0$ which is a contradiction.
\end{proof}

\subsection{The Modified Dimension and Exact Sequences}\label{SS:moddimandexactseq} \emph{Throughout this section we assume $\cat$ is an abelian category and $d_{V}: V^{*}\otimes V \to \unit$ is an epimorphism for all objects $V$.}  Recall from Section~\ref{SS:ribboncats} that in this setting the tensor functor is necessarily exact.  We will now show that when $J$ is an ambidextrous object $\ideal_J$ and $\md_J$ have a meaning in terms of the splitting of certain exact sequences in $\cat$.  In the specific setting of modular representations of finite groups similar results were first considered by Okuyama \cite{Ok}, Carlson and Peng \cite{CP}, and Benson and Carlson \cite{BC}.   In particular, in that setting an object in $P \in \ideal_{V}$ is called \emph{V-projective} by the authors of \cite{Ok, CP}.  The following results show that their techniques apply in the general setting of ribbon Ab-categories.  

\begin{lemma}\label{L:PSplit}
Let $V,W\in \ob$, $P\in \ideal_V$,  and let $h:P\rightarrow W$ be a morphism.  If  $U \xrightarrow{g} W \rightarrow 0$ is an exact sequence such that $U\otimes V \xrightarrow{g\otimes \Id_V} W\otimes V \rightarrow 0$ is split then there exists a morphism $\hat{h}:P\rightarrow U$ such that $g\circ \hat{h}=h$.
\end{lemma}
\begin{proof}
To prove the lemma we will show that $ \Hom_\cat(P,U) \xrightarrow{g_*}  \Hom_\cat(P,W)$ is onto.   The map $\Hom_\cat(P,W)\rightarrow \Hom_\cat(\unit,W\otimes P^*)$ given by $f\mapsto (f\otimes \Id_{P^*})\circ b_{P}$ is invertible (the inverse given by $k\mapsto (\Id_W\otimes d_P)\circ(k\otimes \Id_P))$.   Using these maps we have the following commutative diagram:
$$ \xymatrix
 { \Hom_\cat(P,U) \ar[rr]^{g_*} \ar[d] & & \Hom_\cat(P,W)  \\
\Hom_\cat(\unit,U \otimes P^*) \ar[rr]^{(g\otimes \Id_{P^*})_*}  & & \ar[u]\Hom_\cat(\unit,W \otimes P^*).}$$
Therefore, it suffices to show $(g\otimes \Id_{P^*})_*$ is onto.  

With this in mind let us prove the following claim: the exact sequence $U\otimes P^* \xrightarrow{g\otimes \Id_{P^*}} W\otimes P^* \rightarrow 0$ is split.   Since $P\in \ideal_V$, Lemma \ref{L:Ideal} implies that $P^*\in \ideal_V$.  So there exists an object $X$ and morphisms $\alpha: P^*\rightarrow V\otimes X$ and $\beta: V\otimes X \rightarrow P^*$ such that $\beta\circ\alpha =\Id_{P^*}$.   We have the following commutative diagram:
$$ \xymatrix
 {U\otimes V\otimes X \ar[rr]^{g \otimes \Id_{V}\otimes \Id_{X}} \ar[d]^{\Id_U\otimes \beta} & & W \otimes V\otimes X  \ar[r] \ar[d]^{\Id_W\otimes \beta}& 0 \\
U\otimes P^* \ar[rr]^{g\otimes \Id_{P^*}}  & & W\otimes P^* \ar[r] &0.}$$
The hypothesis of the lemma imply that the top horizontal map splits, i.e. there exists a morphism $k:W \otimes V\otimes X\rightarrow U\otimes V\otimes X$ such that $(g \otimes \Id_{V}\otimes \Id_{X}) k=\Id_{U\otimes V\otimes X}$.  Let $\hat{k} = (\Id_U\otimes \beta)k(\Id_W\otimes \alpha)$.  Then 
\begin{align*}
(g\otimes \Id_{P^*})\hat k &=(g\otimes \Id_{P^*})(\Id_U\otimes \beta)k(\Id_W\otimes \alpha)\\
&=(\Id_{W}\otimes\beta) (g \otimes \Id_{V}\otimes \Id_{X})k(\Id_W\otimes \alpha)=\Id_{W\otimes P^*}
\end{align*}
and the claim follows.

The morphism $\hat{k}_*:\Hom_\cat(\unit, W\otimes P^*)  \rightarrow \Hom_\cat(\unit,U\otimes P^*)$ is the right inverse to $(g\otimes \Id_{P^*})_*$ and so $(g\otimes \Id_{P^*})_*$ is onto.  Thus, $g_*$ is onto and the result is proved.
\end{proof}

\begin{theorem}\label{T:CanSurIdeal}
Let $V$ and $W$ be any objects of $\cat$ such that $V\in \ideal_W$.  The canonical epimorphism 
\begin{equation}\label{E:canVW}
V^*\otimes V\otimes W \xrightarrow{d_V\otimes \Id_W} W \rightarrow 0
\end{equation}
is split if and only if $\ideal_V=\ideal_W$.
\end{theorem}
\begin{proof}
If the sequence in Equation \eqref{E:canVW} splits then $W\in \ideal_V$.  Since $V\in \ideal_W$ we have $\ideal_W=\ideal_V$.  On the other hand, suppose $\ideal_V=\ideal_W$.   Then there exists an object $X$ and morphisms $\alpha: W\rightarrow V\otimes X$ and $\beta: V\otimes X \rightarrow W$ such that $\beta\circ\alpha =\Id_{W}$.   The sequence 
$$V\otimes V^*\otimes V\otimes W \xrightarrow{\Id_V\otimes d_V\otimes \Id_W}V\otimes W \rightarrow 0$$
 is split by the morphism $b_V\otimes \Id_V\otimes \Id_W$.   Therefore, we can apply Lemma \ref{L:PSplit} to $h=\beta$ and $g=d_V\otimes \Id_W$ to obtain a morphism $\hat h: V\otimes X \rightarrow V^*\otimes V \otimes W$ such that $(d_V\otimes \Id_W) \circ \hat h = \beta$, i.e. $\hat h \circ \alpha$ provides a splitting for \eqref{E:canVW}. 
\end{proof}

\begin{corollary}\label{C:splitd}  Assume $J$ is an ambidextrous object and $V$ is a absolutely simple object in $\ideal_J$ then 
$$V^*\otimes V \otimes J\xrightarrow{d_V\otimes \Id_J} J\rightarrow 0$$
splits if and only if $\md_J(V)\neq 0$. 
\end{corollary}
\begin{proof}
From Lemma \ref{L:dim0} we know that $\md_J(V)\neq 0$ if and only if  $\ideal_J=\ideal_V$.  Thus, the  corollary follows from Theorem \ref{T:CanSurIdeal}.

\end{proof}

We note that in the context of modular representations of a finite group, the above corollary is proven in the case of $J=\unit$ (i.e.\ the trivial module) but with the weaker assumption that $V$ is absolutely indecomposable by Benson and Carlson \cite[Theorem 2.1]{BC}.  

\begin{remark}
When $V$ is an arbitrary element of $ \ideal_J$, Theorem \ref{T:dim0} implies that the if direction of Corollary \ref{C:splitd} still hold.
\end{remark}

\subsection{Projective Objects}\label{SS:projectives}  We record a few elementary results on projective objects in $\cat$. Set $\Proj$ to be the full subcategory of projective objects in $\cat$:
\begin{equation}\label{E:Projdef}
\Proj = \left\{\text{projective objects in $\cat$} \right\}.
\end{equation}   It is straightforward to verify that $\Proj$ is an ideal in
$\cat$.

\begin{lemma}\label{L:projectives}  For any ideal $\ideal$ of $\cat $ one has 
\[
\Proj  \subseteq \ideal.
\] Furthermore, an object $V$ of $\cat$ is projective if and only if  
\[
\ideal_{V} = \Proj.
\]
\end{lemma}
\begin{proof} 
  Let $U\in \ideal$.  As in the previous result, one has an epimorphism
\[
U^*\otimes U \otimes P\xrightarrow{d_U\otimes \Id_P} P\rightarrow 0
\] 
for any object $P$.  When $P$ is projective this epimorphism necessarily
splits. Thus every projective object is an object in $\ideal$.  This
proves the first statement. For the second statement of the lemma, we observe
that when $\ideal_{V}=\Proj $ then, since $V \in \ideal_{V}$, $V$ is clearly
projective.  On the other hand, since the tensor product is an exact functor
it follows that if $V$ is a projective, then $V \otimes M$ (and, hence, any
direct summand) is projective for any object $M$ of $\cat$.  Thus if $V$ is
projective, then necessarily $\ideal_{V}$ consists of only projectives.
\end{proof}

The previous lemma implies that if $J$ is ambidextrous and $P$ is in $\Proj$, then one can always consider $\md_{J}(P)$.  Theorem~\ref{T:dim0} immediately implies the following result.  
\begin{corollary}\label{C:mdvanishesonproj}  Let $J$ be an ambidextrous object in $\cat$ which is not projective.  Then 
\[
\md_{J}(P)=0
\]
for all $P$ in $\Proj$.

\end{corollary}

\section{The ground ring $K$}\label{S:ComRing}  As stated at the beginning, for convenience we assumed $K=\End_{\cat}(\unit )$ is a field.   In general $K$ is only a commutative ring.   For example, this occurs when $\cat$ is the finite dimensional representations of a Drinfeld-Jimbo quantum group defined over $\C [[h]]$.  However an examination of the proofs will confirm that the results presented so far hold for general $K$ once one makes suitable modifications.   For example, in Lemma~\ref{P:ambdim}\eqref{LI:amb2} the condition that $\md_J(V)\neq 0$  should be revised to the condition that  $\md_J(V)$ is not a zero divisor in $K$.  In Section~\ref{S:ModDim} the condition that $\md_J(V)\neq 0$  should be replaced with the condition that  $\md_J(V)$ is invertible in $K$.  The interested reader should have no difficulty in obtaining the general results.

\section{Representations of Lie superalgebras and the Generalized Kac-Wakimoto Conjecture}\label{S:Liesuperalgebras}   In this section we consider the ribbon category $\cat$ given by finite dimensional representations of a Lie superalgebra $\fg$ which are semisimple over $\fg_{\0}$.  As we explain, the tensor product and duality are given by the usual coproduct and antipode on $\fg$.  The braiding is given by the graded flip map.  See Section~\ref{SS:Liesuperalgs} for details on the ribbon category structure on $\cat$.  In this setting we will see that the modified dimension function generalizes superdimension and is closely related to the combinatorially defined notions of defect and atypicality.  In particular, it provides a new point of view on a conjecture of Kac and Wakimoto. 

\subsection{Representations of Lie superalgebras and Atypicality }\label{SS:Liesuperalgs} Recall that a Lie superalgebra $\fg = \fg_{\0} \oplus \fg_{\1}$ is a $\Z_{2}$-graded complex vector space with a bilinear map $[\; , \; ] : \fg \otimes \fg  \to \fg$ which satisfies graded versions of the conditions on a Lie algebra bracket.  A $\fg$-supermodule is a $\Z_{2}$-graded finite dimensional complex vector space $M=M_{\0} \oplus M_{\1}$ which admits a graded action by $\fg$ and satisfies graded versions of the conditions on a Lie algebra module.  The finite dimensional representations of Lie superalgebras have been the object of intense study for over thirty years.  We refer the reader to \cite{Kac1, Kacnote} for background on Lie superalgebras and their representations.

The category $\cat $ in this context will be the category of all finite
dimensional $\fg$-supermodules 
whose restriction to $\fg_{\0}$ is semisimple.  If $M$ and $N$ are two
objects of $\cat$, then a morphism $f: M \to N$ is a linear map which
preserves the $\Z_{2}$-grading in the sense that $f(M_{r}) \subseteq N_{r}$
for $r \in \Z_{2}$ and which satisfies $f(xm)=xf(m)$ for all $x \in \fg$ and
$m \in M$.  Let $M^{*}$ denote the usual linear dual.  That is, $M^{*} =
\Hom_{\C }(M,\C )$ where we declare $\C$ to be $\Z_{2}$-graded by being
concentrated in degree $\0$ and then an element of $M^{*}$ is of degree $\0$
if it preserves the $\Z_{2}$-grading and of degree $\1$ if it reverses the
grading. The action on $M^{*}$ is given by $(xf)(m) = -(-1)^{\overline{x}
  \cdot \overline{f}}f(xm)$ where $x \in \fg$ and $f \in M^{*}$ are assumed to
be homogenous and where 
here and elsewhere we write $\overline{a} \in \Z_{2}$
for the degree of a homogeneous element $a$. We also use the convention 
that a formula is given only on homogenous elements and the
general case is given by applying linearity.  Given $M$ and $N$ in $\cat$, the
tensor product is given by $M \otimes N = M \otimes_{\C} N$ as a vector space.
The $\Z_{2}$-grading is given by the formula $\overline{m \otimes n} =
\overline{m}+\overline{n}$ for all homogeneous $m \in M$ and $n\in N$.  The
action of $\fg$ is given by the formula
\[
x.(m\otimes n) = (x.m) \otimes n + (-1)^{\p{x}\;\p{m}} m \otimes (x.n)
\]
for all homogenous $x \in \fg$, $m \in M$, and $n \in N$.   The unit object is then the trivial supermodule $\C$ (which is concentrated in degree $\0$ in the $\Z_{2}$-grading).

The ribbon category structure on $\cat$ is given as follows.  Given an object $V \in \cat$ the map $b_{V}: \C  \to V \otimes V^{*}$ is given by $1 \mapsto \sum_{i} v_{i} \otimes v_{i}^{*}$ where $\{v_{i} \}$ is a homogeneous basis for $V$ and $v_{i}^{*}\in V^{*}$ is defined by $v_{i}^{*}(v_{j}) = \delta_{i,j}$.  The map $d_{V}: V^{*}\otimes V \to \C$ is the evaluation map $f \otimes v \mapsto f(v).$  The braiding $c_{V,W}: V \otimes W \to W \otimes V$ is given by $v \otimes w \mapsto (-1)^{\overline{v} \cdot \overline{w}}w \otimes v$.  Observe that $c_{W,V} \circ c_{V,W} = \Id_{V \otimes W}$ for all objects $V,W \in \cat$ and hence by definition $\cat$ is symmetric.  The twist maps are the identity.

A Lie superalgebra $\fg$ is said to be \emph{basic} if it admits a nondegenerate even invariant bilinear form.  The simple basic Lie superalgebras were classified by Kac \cite{Kac1} and are, in the notation of \emph{loc. cit.}, the Lie superalgebras of type $A(m,n)$, $B(m,n)$, $C(n)$, $D(m,n)$, $D(2,1;\alpha)$, $F(4)$, and $G(3)$.  Note that these are also \emph{classical} as $\fg_{\0}$ is a reductive Lie algebra.  \emph{Throughout we will assume that our Lie superalgebras are basic and classical}.

Let $\fg$ be a basic classical Lie superalgebra and fix $\ft$ to be 
a maximal torus contained in $\g_{\0}$. Also fix a choice of Borel subalgebra $\fb \subset \fg$ which contains $\ft$.  Having done so, one can assign to each simple $\fg$-supermodule in $\cat$ a highest weight $\lambda \in \ft^{*}$ and we write $L(\lambda)$ for the simple supermodule labelled by $\lambda \in \ft^{*}$.  

Let $\Phi$ be the set of roots with respect 
to $\ft$. We have that $\Phi=\Phi_{\0}\sqcup \Phi_{\1}$ 
where $\Phi_{\0}$ (resp.\ $\Phi_{\1}$) is the set of even roots (resp.\ odd roots). 
The positive roots will be denoted by $\Phi^{+}$ and the negative roots by $\Phi^{-}$. 
Set $\Phi_{\0}^{\pm}=\Phi_{\0}\cap \Phi^{\pm}$ and $\Phi_{\1}^{\pm}=\Phi_{\1}\cap \Phi^{\pm}$. 
The bilinear form on $\fg$ induces a bilinear form on $\ft^{*}$ which we denote by $(\; , \; )$.   Following Kac and Wakimoto \cite[Section~2]{KW} we
define the \emph{defect of $\fg$}, denoted by $\defect(\fg),$ to be the 
dimension of a maximal isotropic subspace in the $\mathbb{R}$-span of $\Phi$.     

Let $\lambda\in \ft^{*}$. The \emph{atypicality} of $\lambda$, denoted $\atyp (\lambda)$, is the maximal number of linearily independent, mutually orthogonal, positive isotropic roots $\alpha \in  \Phi^{+}$ such that $ (\lambda+\rho,\alpha)=0,$
where $\rho=\frac{1}{2}(\sum_{\alpha\in \Phi_{\0}^{+}} \alpha- \sum_{\alpha\in \Phi_{\1}^{+}} \alpha)$. Note that by definition
\[
\atyp(\lambda)\leq \defect(\fg).
\]  Given a simple supermodule $L(\lambda)$ we define the atypicality of $L(\lambda)$ by $\atyp (L(\lambda)) = \atyp (\lambda)$.  Note that this definition is known  to be independent of the choice of $\ft$ and $\fb$ and, hence, an invariant of the simple supermodule and not the choice of parameterization.  In particular, a simple $\fg$-supermodule $L$ is called  \emph{typical} if $L$ has atypicality zero.  If $L$ is a typical supermodule then $L$ is a projective object  in $\cat$ by \cite[Theorem 1]{Kacnote} and so by Lemma~\ref{L:projectives} $\ideal_L$ is the ideal of projective $\fg$-supermodules $\Proj$.

Given a finite dimensional $\fg$-supermodule, $M$, the \emph{superdimension} of $M$ is the integer 
\[
\sdim (M) = \dim (M_{\0})-\dim (M_{\1}).
\]  

\subsection{The trace on $\Proj$ for Lie superalgebras of Type A or C} \label{SS:TraceProj}
Let $\fg$ be $\glmn$ or a simple Lie superalgebra of Type A or C.  Lemma 2.8 of \cite{GP2} states that we can fix a typical simple $\fg$-supermodule $J$ whose tensor product with itself is semisimple and multiplicity free.  By Remark~\ref{R:ambi} we have that $J$ is ambidextrous and by Theorem~\ref{T:inducedtrace} this induces an ambidextrous trace on $\ideal_{J}=\Proj$.  On the other hand, in \cite{GP4} the first and third authors defined a linear map $\mt_{V}:\End_{\cat} (V) \to  \End_{\cat }(\unit )$ for each typical simple supermodules $V$ using quantum groups and low dimensional topology.  In this subsection we show these two notions coincide and use this result to give an explicit formula for the modified dimension defined on this ideal.

We will now recall a trace which is defined in \cite{GP4}.   Let $\md:\{\text{typical supermodules}\} \rightarrow \C$ be the function defined
  by
  \begin{equation}\label{E:FormulaMD}
    \md(L({\lambda}))=\prod_{\alpha\in\Phi_{\p0}^+} \frac{(\lambda
      +\rho,\alpha)}{(\rho,\alpha)}  
    \Big/ \prod_{\alpha\in \Phi_{\p1}^+}(\lambda +\rho,\alpha)
  \end{equation}
  for any typical supermodule $L(\lambda)$.   Note $\md(L({\lambda}))$ is non-zero for any typical supermodule $L(\lambda)$.  Let $V\in\Proj$ and $f\in\End_\cat(V)$.  Choose a typical supermodule $V_0$.  Then since $\ideal_{V_0}=\Proj$ there exists $W\in \cat$ and morphisms 
  $\alpha\in\Hom_\cat(V_0\otimes W,V)$ and $\beta\in\Hom_\cat(V,V_0\otimes
  W)$ such that $\alpha\circ \beta=\Id_V$.  Then by Theorem 1 of \cite{GP4} we have that the map $\text{str}'_V: \End_{\cat}(V) \to \End_{\cat}(\unit )$ given by 
  $$\text{str}'_V(f)=\md(V_0)\langle\tr_R(\beta \circ f \circ \alpha)\rangle$$
is a well defined linear function depending only on $f$; that is, it does not depends on $V_0$, $W$, $\alpha$ or $\beta$.  Moreover,  the family $\{\text{str}'_V\}_{V \in \Proj }$ is a trace on $\Proj$.   In \cite{GP4} the proof that $\{\text{str}'_V\}$ is a trace uses quantum algebra and low-dimensional topology.

As we mentioned above, Remark \ref{R:ambi} implies that the canonical linear map $\mt=\langle\; \rangle$ is a non-zero ambidextrous trace on $J$.  Let $\{\mt_V\}$ be the unique trace on $ \ideal_J$ determined by $\mt=\mt_J$ (see Theorem \ref{T:ambi!Trace}).  By definition $\text{str}'_J=\md(J)\mt$ (take $V_0=J$, $W=\C$, $\alpha$, $\beta=\alpha^{-1}$, where $\alpha:J\otimes \C\rightarrow J$ is the right unit constraint of $J$).  Then, from the uniqueness of the trace we have $\frac1{\md(J)}\text{str}'_V=\mt_V$ for all $V\in \ideal_J=\Proj$ (note that the factor $\frac1{\md(J)}$ is just a re-normalization constant).  Thus, we recover the trace given in \cite{GP4}.  

Now if $L$ typical then 
\begin{equation}\label{E:moddimfrac}
\md_J(L)=\mt_L(\Id_L)=\frac1{\md(J)}\text{str}'_L(\Id_L)=\frac{\md(L)}{\md(J)}
\end{equation}
where the last equality follows from taking $V_0=L$, $W=\C$ and $\alpha=\beta^{-1}$ (here $\alpha$ is the right unit constraint).   Equation \eqref{E:FormulaMD} now leads to an explicit formula for $\md_J(L)$.  In particular, $\md_J(L)\neq 0$ whenever $J$ and $L$ are both typical.  To summarize, the authors of \cite{GP2} prove the following result.
\begin{theorem}\label{T:typicalcase}  Let $\fg$ be $\glmn$ or a simple Lie superalgebra of Type A or C.  Let $J$ be a typical simple supermodule.  Then $J$ is ambidextrous, $\ideal_{J}=\Proj$, and if $L$ is another typical simple supermodule, then
\[
\md_{J}(L) \neq 0.
\]  Furthermore, $\md_{J}(L)$ can be computed by the explicit formula given by \eqref{E:moddimfrac}.

\end{theorem}

Similarly, we expect to be able to combine the techniques of this paper with the ideas and results of \cite{GP4,GPT} to derive a formula for $\md_J(L)$ where $J$ and $L$ are simple supermodules of any basic classical Lie superalgebra.  The formula would be expressed in terms of the super characters of the supermodules $J$ and $L$.  Also, similar arguments show that the representations of quantum $\sll_{2}$ at a root of unity considered in \cite{GPT} also fit within our framework.

\subsection{The Generalized Kac-Wakimoto Conjecture}\label{SS:GKWconj} We now state an intriguing conjecture of Kac and Wakimoto which gives a representation theoretic interpretation of the combinatorial notions of defect and atypicality.

\begin{conjecture} \cite[Conjecture 3.1]{KW} Let $\fg$ be a simple basic classical Lie superalgebra 
and $L(\lambda)$ be a finite-dimensional simple $\fg$-supermodule. Then 
\[
\atyp (L(\lambda)) =\defect(\fg)
\] if and only if  $\sdim (L(\lambda))\neq 0$.  
\end{conjecture} The authors of \cite{KW} give numerous examples where the conjecture holds, including what they call tame representations (which include, for example, the polynomial representations of $\glmn$).  In \cite[Lemma 7.1]{DS} Duflo and Serganova prove for contragradiant Lie superalgebras that if the atypicality is strictly less than the defect, then the superdimension must be zero.  This direction of the conjecture was also verified for $\glmn$ by the authors of \cite{BKN2} using the support varieties they introduced in \cite{BKN1}.  Recently Serganova has announced a proof for the classical contragradiant Lie superalgebras using category equivalences, Zuckerman functors, and a character formula of Penkov \cite{Ser}.

The approach of the current paper allows us to recast and generalize the above conjecture.  Consider the trivial supermodule $\C$.  A direct computation verifies that $\atyp (\C ) = \defect (\fg )$.  By Example~\ref{Ex:unit}, $\C$ is ambidextrous, $\ideal_{\C}= \cat$, and $$\md_{\C}(M) = \sdim(M)$$ for any $\fg$-supermodule $M$.  Thus, one can rephrase the above conjecture as follows:  For any $M \in \ideal_{\C}$, $\atyp (M) = \atyp (\C )$ if and only if $\md_{\C}(M)\neq 0$.  Our new point of view naturally suggests the following generalized Kac-Wakimoto conjecture. 

\begin{conjecture}\label{C:genKWconj}  Let $\fg$ be a basic classical Lie superalgebra, let $J$ be a simple ambidextrous $\fg$-supermodule and let $L \in \ideal_{J}$ be a simple $\fg$-supermodule.  Then $$\atyp(L) = \atyp(J)$$ if and only if $\md_J(L) \neq 0$.
\end{conjecture}

The available evidence suggests if $J$ is a simple $\fg$-supermodule, then $J$ is ambidextrous and $\ideal_{J}$ contains all simple supermodules whose atypicality is not more than the atypicality of $J$ (cf.\ Theorem~\ref{T:polys}).  If so, the above conjecture can be rephrased in a less cumbersome fashion.

Let us briefly discuss the extreme cases.  When $J=\C$ then $J$ has the maximal possible atypicality, one is reduced to the ordinary Kac-Wakimoto conjecture, and the evidence for it also, of course, supports the generalization.  At the other extreme, assume $\fg$ is $\glmn$ or simple Lie superalgebra of Type A or C and $J$ in $\cat$ is a typical simple supermodule.  Then Theorem~\ref{T:typicalcase} confirms the conjecture here as well.  We consider the intermediate cases in the following sections.

\subsection{Support Varieties}\label{SS:supports}  In \cite{BKN1} cohomological support varieties for classical Lie superalgebras were introduced.  As they will be needed in what follows, we now discuss how they relate to the results of this paper and the generalized Kac-Wakimoto conjecture.   Given a classical Lie superalgebra, $\fg$, let $\fe \subseteq \fg$ denote the detecting subalgebra as given in \cite[Section 4]{BKN1}.   Then, given a $\fg$-supermodule $M$ in $\cat$, one can functorially define varieties $\V_{\fg}(M)$ and $\V_{\fe}(M)$.  Let $\fa$ denote $\fg$ or $\fe$.  By \cite[Section 4.6]{BKN2}, these varieties satisfy 
\begin{align}
\V_{\fa}(M \oplus N) &= \V_{\fa}(M) \cup \V_{\fa}(N),\label{E:varpropplus}\\
\V_{\fa}(M \otimes N) &\subseteq \V_{\fa}(N) \cap \V_{\fa}(N). \label{E:varproptimes}
\end{align}

The following result shows that these support varieties are compatible with the constructions introduced here.
\begin{proposition}\label{P:supp}  Let $\fg$ be a classical Lie superalgebra and let $\fa$ denote $\fg$ or the detecting subalgebra $\fe$.  Say $L,J \in \ob$ and $L \in \ideal_{J}$.
\begin{enumerate}
\item  Then $\V_{\fa}(L) \subseteq \V_{\fa}(J)$.
\item If $J$ is ambidextrous, then $\md_{J}(L)\neq 0$ implies $\V_{\fa}(L)=\V_{\fa}(J)$.
\end{enumerate}
\end{proposition}

\begin{proof}  Since $L$ being an element of $\ideal_{J}$ implies $L$ is a direct summand of $J\otimes X$ for some finite dimensional $\fg$-supermodule $X$, the first statement is an immediate consequence of \eqref{E:varpropplus} and \eqref{E:varproptimes}.  To prove the second statement, since $\md_{J}(L) \neq 0$ implies $\ideal_{L}=\ideal_{J}$ by Theorem~\ref{T:dim0} we have $J \in \ideal_{L}$. The desired equality then follows by part (1).
\end{proof}

Boe, Kujawa, and Nakano conjectured in \cite[Conjecture 7.2.1]{BKN1} that for a simple $\fg$-supermodule $L$ one has
\begin{equation}\label{E:atyp}
\atyp (L) = \dim \left(\V_{\fe}(L) \right),
\end{equation} where here $\dim$ denotes the dimension as an algebraic variety.  We remark that \eqref{E:atyp} is proven for $\glmn$ in \cite[Theorem 4.8.1]{BKN2}.  It is still open in general.

Observe that whenever \eqref{E:atyp} is valid, then Proposition~\ref{P:supp} immediately implies one direction of the generalized Kac-Wakimoto conjecture (cf.\ Theorem~\ref{T:TypeAgenKWconj} for the $\glmn$ case).  
Let us also remark that Duflo and Serganova \cite{DS} defined associated varieties for representations of Lie superalgebras which are different from the support varieties considered here.   The interested reader can verify that the varieties introduced there have properties analogous to \eqref{E:varpropplus} and \eqref{E:varproptimes} and are similarly compatible with the techniques of this paper.

\subsection{The Generalized Kac-Wakimoto Conjecture for $\glmn$}\label{SS:TypeA} \emph{For the remainder of Section~\ref{S:Liesuperalgebras} we assume  $\fg =\glmn$ and, as $\glmn \cong \mathfrak{gl}(n|m)$, that $m \leq n$.}  

Let us set the notation we will use when considering $\glmn$.  We take the
matrix realization of $\glmn$ as $(m+n) \times (m+n)$ matrices, the
$\Z_{2}$-grading is given by setting the $i,j$ matrix unit to be degree $\0$
if $1 \leq i,j \leq m$ or $m+1 \leq i,j \leq m+n$, and $\1$ otherwise, and the
bracket is given by the super commutator.  We choose $\ft$ to be the Cartan
subalgebra given by diagonal matrices and $\fb$ to the Borel subalgebra given
by upper triangular matrices.  Recall that if $1 \leq i \leq m+n$ and
$\varepsilon_{i} \in \ft^{*}$ is the linear functional which picks off the
$i$th diagonal entry of an element of $\ft$, then $\varepsilon_{1}, \dotsc,
\varepsilon_{m+n}$ provides a basis for $\ft^{*}$.  With respect to this
choice the roots are
\[
\{\varepsilon_{i}-\varepsilon_{j} \mid 1 \leq i,j \leq m+n, i \neq j  \},
\]
and the positive roots are 
\[
\{\varepsilon_{i}-\varepsilon_{j} \mid 1 \leq i < j \leq m+n, i\neq j \}.
\]  A root is even if $1 \leq i,j \leq m$ or $m+1 \leq i,j \leq m+n$, and otherwise it is odd. 

Given $\lambda \in \ft^{*}$ we often choose to instead write $(\lambda_{1}, \dotsc , \lambda_{m+n})$ where $\lambda = \sum_{i} \lambda_{i}\varepsilon_{i}$.  For $\glmn$ it is convenient (and harmless) to choose 
\[
\rho = (m-1, \dotsc, 1,0,0,-1,\dotsc ,-(n-1)). 
\] 
The bilinear form on $\ft^{*}$ is given by 
\[
(\varepsilon_{i}, \varepsilon_{j}) = \begin{cases} \delta_{i,j}, & 1 \leq i \leq m; \\
                                                    -\delta_{i,j}, & m+1 \leq i \leq m+n.
\end{cases}
\]

For $\glmn$ the equality given in \eqref{E:atyp} is proven in \cite[Theorem 4.8.1]{BKN2}.  Therefore we can prove one direction of the generalized Kac-Wakimoto conjecture for $\glmn$.  
\begin{theorem}\label{T:TypeAgenKWconj}  Let $\fg=\glmn$, let $L$ and $J$ be simple $\fg$-supermodules with $J$ ambidextrous and $L \in \ideal_J$.   Then $\atyp (L) \leq \atyp (J)$.  Furthermore, if $\md_J(L) \neq 0$, then $\atyp(L)=\atyp(J)$.
\end{theorem}
\begin{proof}  Since $L \in \ideal_{J}$, combining part (1) of Proposition~\ref{P:supp} and \cite[Theorem 4.8.1]{BKN2}, one obtains $\atyp (L) \leq \atyp (J)$.  By part (2) of Proposition~\ref{P:supp}, if $\md_{J}(L) \neq 0$ then $\V_{\fe}(L)=\V_{\fe}(J)$.  Again applying \cite[Theorem 4.8.1]{BKN2} we have  
\[
\atyp (L) = \dim  \left(\V_{\fe}(L) \right)= \dim \left(\V_{\fe}(J) \right) =\atyp (J),
\] as required.
\end{proof}

\subsection{Polynomial Representations of $\glmn$}\label{SS:TypeAambi}  The main obstacle to using the theory discussed in this paper on a given category is finding a sufficient number of ambidextrous objects.  In the case of $\glmn$, however, the polynomial representations give a ready source of ambidextrous objects.  Our main result of this section is Theorem~\ref{T:polys} where we show that all simple polynomial representations of $\glmn$ are ambidextrous and that the generalized Kac-Wakimoto conjecture holds for these representations.

Let $V$ denote natural representation of $\glmn$.  A simple $\glmn$-supermodule is said to be a \emph{polynomial representation of degree $d$} if it appears as a composition factor in $V^{\otimes d}$.  More generally, a supermodule is said to be \emph{polynomial} if all its composition factors are polynomial.  In the case of $\glmn$ this class of representations were first studied by Berele-Regeev \cite{BR} and Sergeev \cite{Se}.  A key result is that the full subcategory consisting of all polynomial representations of $\glmn$ is known to be a semisimple category; in particular, the tensor product of polynomial representations is always completely reducible.  

Let us recall the classification of the simple polynomial $\glmn$-supermodules. To do so requires recalling certain basic notions regarding partitions.  Recall that a \emph{partition} of $d$ is a weakly decreasing sequence of nonnegative integers which sum to $d$.  One can visually represent a partition $\gamma$ via its \emph{Young Diagram}: the diagram obtained by having $\gamma_{i}$ boxes (or nodes) in the $i$th row with all rows left justified.  We refer to a box in the $i$th row and $j$th column of the Young diagram as the $(i,j)$ node of the diagram.  If $\gamma$ is a partition, then we write $\gamma^{T}$ for the transpose partition (i.e.\ the partition obtained by reflecting the Young diagram of $\gamma$ across the line through the nodes along the $(i,i)$ diagonal).   Given two finite sequences of integers $\gamma_{1}$ and $\gamma_{2}$, we write $\gamma_{1}\# \gamma_{2}$ for the concatenation of the two sequences. A partition $\gamma$ is said to be a \emph{$(m,n)$ hook partition} if the $(m+1,n+1)$ node is \emph{not} a node of the Young diagram of $\gamma$.  

The simple polynomial $\glmn$-supermodules are parameterized by the set of $(m, n)$ hook partitions.  We write $L_{\gamma}$ for the simple supermodule labeled by the $(m,n)$ hook partition $\gamma$.  On the other hand, as discussed in Section~\ref{SS:Liesuperalgs}, the simple supermodules can be labelled by highest weight with respect to our choice of Cartan and Borel subalgebras.  We write $L(\lambda)$ for the simple supermodule labeled by $\lambda \in \ft^{*}$.  The translation between the two labelings is given as follows. 

   In terms of the highest weight parameterization of the simple supermodules, $L(\lambda)$ is a polynomial representation if and only if the $\lambda = (\lambda_{1}, \dotsc , \lambda_{m+n})$ is a sequence of integers such that $(\lambda_{1}, \dotsc , \lambda_{m})$ and $(\lambda_{m+1}, \dotsc , \lambda_{m+n})$ are partitions and the sequence 
\begin{equation}\label{E:polypartition}
\tau(\lambda) = (\lambda_{1}, \dotsc ,\lambda_{m}) \# (\lambda_{m+1}, \dotsc , \lambda_{m+n})^{T}
\end{equation}
is a partition.  Equivalently, the representation $L(\lambda)$ is polynomial if $(\lambda_{1}, \dotsc , \lambda_{m})$ and $(\lambda_{m+1}, \dotsc , \lambda_{m+n})$ are partitions and the number of nonzero parts of $(\lambda_{m+1}, \dotsc ,\lambda_{m+n})$ does not exceed $\lambda_{m}$.  It is straightforward to verify that the partitions given by \eqref{E:polypartition} are precisely the $(m,n)$ hook partitions.

One can easily use \eqref{E:polypartition} to go between this parameterization and the one by highest weight.  Namely, we have 
\[
L(\lambda) = L_{\tau(\lambda)}.
\]  Note that $\tau$ is an involution.

In the $(m,n)$ hook partition parameterization the character of $L_{\gamma}$ is given by the hook Schur function $HS_{\gamma}$.  By definition, if $\gamma$ is a partition which is not a $(m,n)$ hook function, then $HS_{\gamma} =0$.  Recalling that the tensor product of representations corresponds to the product of their characters, the following result of Remmel \cite[(1.10)]{Re} will be essential.

\begin{lemma}\label{L:productofhookschurfunctions}  Let $\gamma_{1}$ and $\gamma_{2}$ be two $(m,n)$ hook partitions.  Let $S_{\gamma_{1}}$ and $S_{\gamma_{2}}$ be the ordinary Schur functions labeled by the partitions $\gamma_{1}$ and $\gamma_{2}$, respectively.  If 
\[
S_{\gamma_{1}}S_{\gamma_{2}} = \sum_{\mu} g_{\gamma_{1}, \gamma_{2}}^{\mu} S_{\mu},
\] then 
\[
HS_{\gamma_{1}}HS_{\gamma_{2}} = \sum_{\mu} g_{\gamma_{1}, \gamma_{2}}^{\mu} HS_{\mu},
\] where the sum is over all partitions $\mu$.  That is, the product of hook Schur functions is given by the ordinary Littlewood-Richardson rule for $\mathfrak{gl}(m+n)$.
\end{lemma}

We call a partition $\gamma$ a \emph{rectangle} if the nonzero parts of $\gamma$ are all of equal size.  Going back to Kostant, it is known that the product $S_{\gamma}S_{\gamma}$ is multiplicity free (i.e.\ $g_{\gamma,\gamma}^{\mu} \leq 1$ for all partitions $\mu$) for any rectangular partition $\gamma$ (cf.\  \cite[Theorem 3.1]{St}).  Combining this with the above lemma yields the following key result.

\begin{theorem}\label{T:rectangelsareambi}  If $\gamma$ is a $(m,n)$ hook
  partition which is a rectangle then $L_\gamma$ is ambidextrous.  
\end{theorem}

\begin{proof}  Since $L_{\gamma}$ is a polynomial representation, $L_{\gamma} \otimes L_{\gamma}$ is semisimple and its composition factors can be determined by considering the product 
\[
HS_{\gamma}HS_{\gamma} = \sum_{\mu} g_{\gamma,\gamma}^{\mu} HS_{\mu}.
\]  However, by Lemma~\ref{L:productofhookschurfunctions} the coefficients $g_{\gamma,\gamma}^{\mu}$ are given by the ordinary Littlewood-Richardson rule.  Since $\gamma$ is a rectangle this product is multiplicity free by \cite[Theorem 3.1]{St}.  This implies that $\End_{\cat}(L_{\gamma}\otimes L_{\gamma})$ is a commutative ring and so by part (1) of Lemma~\ref{P:ambdim} any linear map provides an ambidextrous trace on $\End_{\cat}(L_{\gamma})$.
\end{proof}

We remark that an immediate consequence of the previous theorem is that there are simple $\glmn$-supermodules of every possible degree of atypicality which are ambidextrous (e.g.\ the representations $L_{\sigma(k)}$ defined just before Theorem~\ref{T:polys}).

Before proving the main result of this section, we first prove a preparatory lemma.  In the following proof we use results on translation functors for $\glmn$ which can be found in \cite{brundan,kujawa}.  First used in this setting by Serganova, these translation functors are given by tensoring with either the natural representation or its dual and then projecting onto a block.  Using these translation functors one can create a colored, directed graph as follows.  The nodes are labelled by the simple $\glmn$-supermodules.  A directed edge goes from $L(\lambda)$ to $L(\mu)$ if $L(\mu)$ appears in the socle of $L(\lambda) \otimes V$.  The edge is colored by an integer determined by the blocks which contain $L(\lambda)$ and $L(\mu)$. 

In \cite[Theorem 2.5]{kujawa} it was shown that this directed graph is described by the combinatorics of a crystal graph (in the sense of Kashiwara) associated to a certain representation of the Kac-Moody algebra $\mathfrak{gl}_{\infty}(\C )$.  We refer the reader to \cite[Theorem 4.36, Section 3.d]{brundan} for both the statement of the necessary result and an explicit description of the crystal graph. 

If $\lambda$ and $\mu$ are two tuples of nonnegative integers, then we write $\lambda \subseteq \mu$ if $\lambda_{i} \leq \mu_{i}$ for all $i >0$.

\begin{proposition}\label{P:polyideals}  Let $L_{\lambda}$ and $L_{\mu}$ be polynomial representations of $\glmn$ and assume that $\lambda \subseteq \mu$ and $\sum_{i} (\mu_{i}-\lambda_{i}) =1$; that is, that the diagram for $\mu$ can be obtained from the diagram of $\lambda$ by adding a single node.  Then,
\[
\ideal_{L_{\lambda}} \subseteq \ideal_{L_{\mu}}.
\]

Furthermore, if $L_{\lambda}$ and $L_{\mu}$ have the same degree of atypicality, then, 
\[
\ideal_{L_{\lambda}} = \ideal_{L_{\mu}}.
\]
\end{proposition}

\begin{proof}  Let $V$ be the natural representation for $\glmn$ which by definition has character equal to $HS_{(1)}$. Thus to compute the composition factors of $L_{\lambda} \otimes V$ it suffices to compute 
\[
HS_{\lambda}HS_{(1)} = \sum_{\gamma} g_{\lambda,(1)}^{\gamma}HS_{\gamma}.
\]  By Lemma~\ref{L:productofhookschurfunctions} the coefficents $g_{\lambda, (1)}^{\gamma}$ are given by the ordinary Littlewood-Richardson rule.  However it is well known in this case that $g_{\lambda, (1)}^{\gamma}$ is zero or one.  It is one if and only if $\lambda \subset \gamma$ and $\gamma$ is obtained from $\lambda$ by the addition of a single node  \cite[I.3, Exercise~11]{macdonald}.  By the assumptions on $\lambda$ and $\mu$, it follows that $g_{\lambda, (1)}^{\mu}=1$.  Since $L_\lambda\otimes V$ is a polynomial representation, it is completely reducible.  Taken together this implies $L_{\mu}$ is a direct summand of $L_{\lambda} \otimes V$ and, hence, $\ideal_{L_{\mu}} \subseteq \ideal_{L_{\lambda}}$.

Furthermore, in the crystal graph language of \cite{brundan, kujawa} there is a directed edge from $L_{\lambda}$ to $L_{\mu}$ colored with some integer $a \in \Z$ and $L_{\mu}=F_{a}L_{\lambda}$.  Since $F_{a}L_\lambda$ is simple, by \cite[Theorem 4.36(i)]{brundan} $L_{\mu}$ is at the end of this directed string of color $a$.  However, in this crystal graph the $a$-strings are of length at most two.  See \cite[Section 3.d]{brundan} for a case by case description of the possible $a$-strings which can occur.  The key observation is that there is only one $a$-string of length two, namely case (4*) in \emph{loc. cit.} However, this case is excluded here as the second to last and last nodes of this $a$-string have differing atypicality and, by assumption, $L_{\lambda}$ and $L_{\mu}$ have the same atypicality.
 Therefore, $L_{\lambda}$ and $L_{\mu}$ form an $a$-string of length precisely one.  By \cite[Theorem 4.36(ii)]{brundan} one has that $E_{a}L_{\mu} \cong  L_{\lambda}$.  That is, that $L_{\lambda}$ is a direct summand of $L_{\mu}\otimes V^{*}$.  Hence, $\ideal_{L_{\lambda}} \subseteq \ideal_{L_{\mu}}$.  This proves the desired equality.

\end{proof}

Now for each $k=0, \dotsc ,m$ let $L_{\sigma(k)}$ denote the polynomial representation labelled by the $(m,n)$ hook partition 
\[
\sigma(k)=(n-k, \dotsc, n-k, 0, \dotsc ,0),
\] where there are precisely $m-k$ entries equal to $n-k$.  Since $\sigma(k)$ is a rectangle, by Theorem~\ref{T:rectangelsareambi} $L_{\sigma(k)}$ is an ambidextrous object for all $k=0, \dotsc , m$.  Also note that since $\sigma(k)_{j}=0$ for $j=m+1, \dotsc m+n$, that one has 
\begin{equation*}
L(\sigma(k)) = L_{\sigma(k)}.
\end{equation*}
By direct calculation, $L_{\sigma(k)}$ has atypicality $k$.  A simple inductive argument using Proposition~\ref{P:polyideals} shows that 
\begin{equation}\label{E:rectanglecontainments}
\ideal_{L_{\sigma(l)}} \subseteq \ideal_{L_{\sigma(k)}}
\end{equation} whenever $l \leq k$.

\begin{theorem}\label{T:polys}  Let $L_{\lambda}$ be the simple polynomial representation of $\glmn$ labelled by the $(m,n)$ hook partion $\lambda$.  Assume that $L_{\lambda}$ has atypicality $k$.  Then the following statements hold true.
\begin{enumerate}
\item One has $\ideal_{L_{\lambda}} = \ideal_{L_{\sigma(k)}}$.
\item One has $\md_{L_{\sigma(k)}}(L_{\lambda}) \neq 0$
\item All polynomial representations of $\glmn$ are ambidextrous.
\item If $L_{\nu}$ is a polynomial representation of $\glmn$ of atypicality less than or equal to the atypicality of $L_{\lambda}$, then $L_{\nu} \in \ideal _{L_{\lambda}}$ and $\md_{L_{\lambda}}(L_{\nu}) \neq 0$ if and only if the atypicality of $L_{\lambda}$ equals the atypicality of $L_{\nu}$.  That is, the Generalized Kac-Wakimoto Conjecture holds for polynomial representations of $\glmn$.
\end{enumerate}
\end{theorem}

\begin{proof} Let $\lambda = (\lambda_{1}, \lambda_{2}, \dotsc) $ be a $(m,n)$ hook partition, let $\mu = \tau(\lambda)$ be the highest weight of $L_{\lambda}$, and say $\mu$ has atypicality $k$.   As we will be interested in keeping close track of atypicality and as this is computed using the highest weight of a representation, we will mainly label simple supermodules with their highest weight.   

 Before proving the theorem we first obtain some information about which entries of $\mu$ contribute to its atypicality.  Let 
\[
t=\lambda_{m+1}.
\]
  We first observe that if $1 \leq i \leq m$, $m+1 \leq j \leq m+n$, and $(\mu + \rho, \varepsilon_{i}-\varepsilon_{j})=0$ then necessarily $j > m+t$.  For if not, then $m+1 \leq j \leq m+t$ and using that the entries of $(\mu_{1}, \dotsc , \mu_{m})$ are weakly decreasing and the fact that $j \leq m+t$ where $t=\lambda_{m+1} \leq \lambda_{m}=\mu_{m}$, we obtain the following contradiction:  
\begin{align}\label{E:mu+rho2}
0 &= (\mu + \rho, \varepsilon_{i}-\varepsilon_{j}) \notag\\
  &= \mu_{i} + m-i + \mu_{j} + m+1 -j \notag \\
   & \geq \mu_{m} + m-i +\mu_{j} + m +1 - j \notag \\
   & \geq \mu_{m} + \mu_{j}+m+1-j \notag \\
   & \geq \mu_{j}+1 \notag \\
   & > 0.
\end{align} 

On the other hand, by \eqref{E:polypartition} we have $\mu_{j}=0$ for $j>m+t$.  Therefore
\begin{equation}\label{E:mu+rho}
(\mu+\rho, \varepsilon_{m+t+l})=t+l-1
\end{equation}
for all $1 \leq l \leq n-t$.   That is, as $l$ runs between $1$ and $n-t$ every integer between $t$ and $n-1$ appears exactly once.  
We also observe that for $1 \leq i \leq m$ we have 
\begin{equation*}
(\mu+\rho,\varepsilon_{i}) \geq (\mu+\rho,\varepsilon_{m})=\mu_m=\lambda_m\geq \lambda_{m+1} =t.
\end{equation*} 

Recall that $k$ was the atypicality of $L(\mu)$.  The previous paragraph implies that  $(\mu + \rho, \varepsilon_{i}-\varepsilon_{j})=0$ for all $i=m-k+1, \dotsc ,m$ and exactly $k$ indices $j \in \{m+t, \dotsc m+n \}$.

We make one other combinatorial observation.  Since the atypicality of $L(\mu)$ is $k$ it follows from \eqref{E:mu+rho}  that $(\mu + \rho, \varepsilon_{m-k})\geq n$.  That is, $\mu_{m-k} \geq n-k$.  Then for $i=1, \dotsc , m-k$ the inequality $\mu_i\geq \mu_{m-k}$ implies $\mu_{i} \geq n - k$.  In particular, since  $\mu_{i}=\lambda_{i}$ for $i=1, \dotsc , m-k$, one has $\sigma(k) \subseteq \lambda$.

We are now prepared to prove statement (1).  We continue to work with the labelling of simple modules by highest weight.  However, observe that if $L(\alpha)$ and $L(\beta)$ are two polynomial representations, then $\alpha \subseteq \beta$ and $\beta$ is obtained from $\alpha$ by the addition of a single node if and only if $\tau(\alpha) \subseteq \tau(\beta)$ and $\tau(\beta)$ is obtained from $\tau(\alpha)$ by the addition of a single node.  Therefore Proposition~\ref{P:polyideals} applies whenever the highest weights of two polynomial representations differ by a single node.

We now show that one can construct a sequence of highest weights all of atypicality $k$,
\[
\sigma(k) = \gamma(1) \subseteq \gamma(2) \subseteq \dotsb \subseteq \gamma(l) = \mu, 
\] so that $\gamma(s)$ is obtained from $\gamma(s-1)$ by the addition of a single node and so that all are the highest weights of polynomial representations. 
The existence of such a sequence - along with the observation in the previous paragraph - immediately implies statement (1) of the theorem via an inductive argument using Proposition~\ref{P:polyideals}.

We do this by proceeding in three stages.  In the first stage, we note that by our earlier observations on the location of atypicality one has that $(\sigma(k) + \rho, \varepsilon_{i})\geq n$ and $(\mu + \rho, \varepsilon_{i})\geq n$ for $i=1, \dotsc , m-k$. 
Consequently we can successively add nodes to the first $m-k$ rows of $\sigma(k)$ in such a way that adding each node yields a partition $\gamma(s)$ with $(\gamma(s) + \rho, \varepsilon_{i}) \geq n$ for all $i=1,...,m-k$.  This implies that $(\gamma(s) + \rho, \varepsilon_{i}-\varepsilon_j)\neq 0$ for all $i=1,...,m-k$ and $j=m+1,\dotsc ,n+m$.  Therefore, the rows $1,...,m-k$ of $\gamma(s)$ are not involved in the calculation of atypicality.  Then since $\gamma(s)_i=\sigma(k)_i =0$ for all $i>m-k$ it follows that $\gamma(s)$ has atypicality $k$.
We continue adding nodes to the first $m-k$ rows until we reach $\gamma(s_{1})$ where $\gamma(s_{1})_{i}=\mu_{i}$ for $i=1, \dotsc m-k$ and $\gamma(s_{1})_{i}=0$ for $i > m-k$.  

In the second stage we start with $\gamma(s_{1})$.  We first observe that if $\alpha$ is such that $\gamma(s_{1}) \subseteq \alpha \subseteq \mu$ and $\alpha_{j}=0$ for $j=m+1, \dotsc, m+n$, then (i) $(\alpha + \rho, \varepsilon_{j}) = 0, -1, \dotsc , -(n-~1)$ for $j=m+1, \dotsc , m+n$, and (ii) we have the inequality $(\alpha + \rho, \varepsilon_{i})\leq (\mu + \rho, \varepsilon_{i})\leq n-1$ for $i=m-k+1,...,m$.  Together these imply we can successively add nodes to rows $m-k+1, \dotsc , m$ of $\gamma(s_{1})$ in such a way that adding each node yields another partition with atypicality $k$.  We continue adding nodes until we reach $\gamma(s_{2})$ where now $\gamma(s_{2})_{i}=\mu_{i}$ for $i=1, \dotsc ,  m$.  

The key observation to make at this point is that if $\alpha$ is such that $\gamma(s_{2}) \subseteq \alpha \subseteq \mu$ with $\alpha_{i} = \mu_{i}$ for $i=1, \dotsc , m$ and $t=\lambda_{m+1}$, then as was explained above using inequality \eqref{E:mu+rho2} (the argument applies equally well by replacing $\mu$ with $\alpha$), $(\alpha + \rho, \varepsilon_{i}-\varepsilon_{j}) > 0$ for all $i=1, \dotsc ,m$ and $j=m+1, \dotsc , m +t$. That is, rows $m+1, \dotsc , m+t$ of $\gamma(s_{2})$ and $\mu$ are not involved in the calculation of the atypicality of $\alpha$.  From this we see we can successively add nodes to these rows $\gamma(s_{2})$ in such a way that adding each node yields a new partition with atypicality $k$.  We continuing adding nodes until we reach $\mu$.  We now have the sequence of weights as desired and, as mentioned above, this suffices to prove statement (1) of the theorem.

Statement (1) implies statement (2) by Theorem~\ref{L:dim0}.  Since $L_{\sigma(k)}$ admits a nonzero ambidextrous trace and $L_{\lambda} \in \ideal_{L_{\sigma(k)}},$ by Theorems~\ref{T:inducedtrace} and~\ref{T:ambi!Trace} $L_{\lambda}$ admits an ambidextrous trace $\mt_{L_{\lambda}}$.  By (2) one has that $0 \neq \md_{L_{\sigma(k)}}(L_{\lambda}) = \mt_{L_{\lambda}}(\Id_{L_{\lambda}})$ which shows that this trace is nontrivial, hence $L_{\lambda}$ admits a nontrivial ambidextrous trace.  This proves (3). 

Finally, to prove (4) one observes that if $l$ is the atypicality of $L_{\nu}$ and $l \leq k$, then by (1) and \eqref{E:rectanglecontainments} we have that $L_{\nu} \in \ideal_{L_{\sigma(l)}} \subseteq \ideal_{L_{\lambda}}$.   Now if $l < k$, then by Proposition~\ref{T:TypeAgenKWconj} one has that $\md_{L_{\lambda}}(L_{\nu}) =0$.  On the other hand, if $l=k$, then by part (1) one has that 
\[
\ideal_{L_{\nu}} = \ideal_{L_{\sigma(k)}} = \ideal_{L_{\lambda}}
\] and by Lemma~\ref{L:dim0} it follows that $\md_{L_{\lambda}}(L_{\nu}) \neq 0$.

\end{proof}

\section{Finite Groups in Positive Characteristic}\label{S:finitegroups}
Fix an algebraically closed field $k$ of characteristic $p>0$.  We will
examine several examples of the above theory applied to representations of
finite groups over the field $k$.  Throughout we will consider the category
$\cat$ of finite dimensional representations over $k$ with the usual tensor
product, dual, etc.  In particular, the trivial module is the unit, $\unit =
k$, and $K=\End_{\cat}(\unit ) \cong k$.
We remark that as $k$ is algebraically closed, an indecomposable module is
absolutly indecomposable and an irreducible module is absolutely irreducible.

\subsection{Cylic group of order $p$}\label{SS:cylicgroup}  Let $C_{p}$ denote the cylic group of order $p$ and $A=kC_{p}$ its group algebra.

\subsubsection{Characteristic $p=2$}\label{SSS:char2}  We first consider the
case $p=2$.  Then there are two indecomposable modules, the trivial module $k$
and $A$ under the left regular action.  By Example~\ref{Ex:unit} trivial
module is automatically ambidextrous therefore we need only consider $A$.  To
be completely concrete, let $g \in C_{2}$ be the cyclic generator, and fix a
basis of $\{v_{1}, v_{2} \} \subset A$ satisfying $gv_{1}=v_{1}$ and
$gv_{2}=v_{1}+v_{2}$. That is, $A$ is 
nonsplit extension of the trivial module with itself and the vector $v_{1}$ spans the unique trivial submodule. 

Since $A$ is projective it follows that $A \otimes A$ is projective and so by dimensions $A \otimes A \cong A \oplus A$.  In terms of our basis we can write the direct sum decomposition as 
\[
A \otimes A = A \oplus A = \langle v_{1}\otimes v_{2}, v_{1}\otimes v_{1}\rangle \oplus \langle v_{1}\otimes v_{2} + v_{2}\otimes v_{2},  v_{1}\otimes v_{2} + v_{2}\otimes v_{1}\rangle,  
\] where $v_{1} \otimes v_{1}$ and $v_{1}\otimes v_{2} + v_{2}\otimes v_{1}$ span the unique trivial submodule of each direct summand.

With respect to the above direct sum decomposition, let 
\[
f: A \otimes A \to A \otimes A
\]
be given by zero on the first direct summand and the identity on the second direct summand.  As usual let $\langle \; \rangle: \End_{\cat}(A) \to k$ denote the canonical map.  Since $v_{1}$ spans the socle of $A$ it follows that for any $\varphi \in \End_{\cat}(A)$ one has $\varphi(v_{1}) = \langle \varphi \rangle v_{1}$.   Using this observation a direct calculation verifies that
\[
\langle \Tr_{L}(f) \rangle = 1 \text{ while } \langle \Tr_{R}(f) \rangle= 0.
\]  Therefore $A$ is \emph{not} ambidextrous as the canonical trace is not an ambidextrous trace.

In terms of ideals we have 
\[
\Proj \subsetneq \ideal_{k} = \cat.
\]  In this case $k$ is the unique ambidextrous indecomposable object in $\cat$.

\subsubsection{Characteristic $p>2$}\label{SSS:charodd} We now consider the
case $p>2$.  The indecomposable representations can be listed as $V_{1},
\dotsc , V_{p}$ where $\dim_{k}\left(V_{r} \right) = r$ for $r=1, \dotsc , p.$
In particular, $V_{1}=k$ and is the unique simple module and $V_{p}=A$ is the
projective cover of $k$ (cf.\ \cite[Chapter II]{Al}).

The modules $V_{1}, \dotsc , V_{p-1}$ all have nonzero categorical dimension.  By Theorem 3.3.1 it then follows that 
\[
\ideal_{V_{i}} = \ideal_{k} = \cat 
\] for $i=1, \dotsc , p-1$.  Furthermore by Lemma~\ref{P:ambdim}(2) these indecomposable modules are ambidextrous.

We now consider $V_{p} \cong A := kC_{p}$, the indecomposable projective cover of the trivial module.  Since $A \otimes A$ is projective we have 
\[
A \otimes A \cong A \oplus \dotsb \oplus A,
\] where there are $p$ direct summands.  We have an action of $\Z_{2}$ on $A \otimes A$ given by the generator of $\Z_{2}$ acting by the morphism $c=c_{A,A} : A \otimes A \to A \otimes A$.  Since $k$ has characteristic different from two, we obtain a decomposition of $A \otimes A$ into the $1$ and $-1$ eigenspaces under this action.  That is, the span of the symmetric and skew symmetric tensors, respectively:
\[
A \otimes A = S^{2}A \oplus \Lambda^{2}A.
\]  Since the braiding is a morphism, this is a decomposition as $C_{p}$-modules. 

We can refine this decomposition as follows. For $i=0, \dotsc , p-1$, let $v_{i} = g^{i} \in A$, so that $v_{0}, \dotsc , v_{p-1}$ forms a basis for $A$.  For $0 \leq i \leq j \leq p-1$ and $k=0, \dotsc , (p-1)/2$, let $E_{k}$ denote the $k$-span of the vectors in the set 
\[
\left\{v_{i}\otimes v_{j} + v_{j} \otimes v_{i} \mid j-i = k \right\}.
\]  A direct calculation verifies that $E_{k}$ is a $C_{p}$-submodule of $S^{2}A$ which is isomorphic to $A$.  Furthermore, we have 
\[
S^{2}A = E_{0} \oplus E_{1} \oplus \dotsb \oplus E_{(p-1)/2}.
\]
Similarly, for $k=1, \dotsc , (p-1)/2$, let $O_{k}$ denote the $k$-span of the vectors in the set 
\[
\left\{v_{i}\otimes v_{j} - v_{j} \otimes v_{i} \mid j-i = k \right\}.
\]  A direct calculation verifies that $O_{k}$ is a $C_{p}$-submodule of $\Lambda^{2}A$ which is isomorphic to $A$.  Furthermore, we have 
\[
\Lambda^{2}A = O_{1} \oplus O_{2} \oplus \dotsb \oplus O_{(p-1)/2}.
\]

Define a morphism $f: A \otimes A \to A \otimes A$ with respect to this direct sum decomposition as follows.  Let $f(O_{k})=0$ for $k=1, \dotsc , p-1$ and $f(E_{k})=0$ for $k=0, 2, \dotsc , p-1$.  The submodule $E_{1}$ is a free module generated by the vector $v_{1}\otimes v_{0} + v_{0} \otimes v_{1}$.  Thus on this summand it suffices to define $f$ by declaring 
\[
f(v_{1}\otimes v_{0} + v_{0} \otimes v_{1}) = v_{2}\otimes v_{1} - v_{1} \otimes v_{2} \in O_{2}.
\]  A direct calculation verifies that 
\[
\Tr_{L}(f)^{t}(v_{0}) = (-1/2)^{t}v_{2t}
\] for $t =1, 2, \dotsc $.  In particular, since $A$ is cyclically generated by $v_{0}$, it follows that $\Tr_{L}(f)^{p}=(-1/2)^{p}\Id_{A}$.  That is, $\Tr_{L}(f)$ is an isomorphism and, hence, 
\begin{equation}\label{E:nonvanishingtrace}
\langle \Tr_{L}(f) \rangle \neq 0.
\end{equation}

However, $f \in \End_{\cat}(A\otimes A)$ is in the $-1$ eigenspace of $\End_{\cat}(A \otimes A)$ under the conjugation action of $c$ (cf.\ Remark~\ref{R:ambi}). In particular, we have 
\[
\Tr_{R}(f) = \Tr_{L}(c \circ f \circ c^{-1}) = -\Tr_{L}(f).
\]  Therefore, as explained in Remark~\ref{R:ambi}, this shows that $A$ is not ambidextrous.

\subsection{The Klein Four Group}\label{SS:kleingroup} Let $k$ be an
algebraically closed field of characteristic two and let $V_{4}$ be the Klein
four group.  There is a single irreducible module, namely the trivial module.
The indecomposable modules were first classified by Ba\u{s}ev \cite{Bv}, and
Heller and Reiner \cite{HR}, but has its roots in work of Kronecker.  It may also be
found in \cite[Theorem 4.3.3]{Ben1} and we use the parameterization therein.
The indecomposable modules come in four types:

\begin{enumerate}
\item The unique projective module $D$ of dimension $4$.
\item $\Omega^{n}(k)$ for $n \in \Z$.  These are the Heller shifts of the
  trivial module and are of dimension $2|n|+1$.  In particular, $\Omega^{0}(k)
  = k$.
\item $V_{n}(\alpha)$ for $n \in \Z_{>0}$ and $\alpha \in k$.  These are of
  dimension $2n$.
\item $V_{n}(\infty)$ for $n \in \Z_{>0}$.  These are of dimension $2n$.
\end{enumerate}

We first describe the ideals generated by the indecomposable modules.  By the
Krull-Schmidt theorem every object in $\cat$ is a direct sum of
indecomposables and hence it suffices to compute the direct summands which
appear in the tensor product of two indecomposable modules.  To do so, we
recall that the representation ring (or Green ring) is defined as the ring
given by taking as elements the isomorphism classes of $kV_{4}$-modules with
addition and multiplication given by direct sum and tensor product,
respectively.  Thus we can make use of the calculation by Conlon \cite{Co} of
the table of products in the representation ring of $V_{4}$ (see also
\cite{Ar} for a summary of Conlon's results). From this one easily reads off
the following result.

\begin{proposition}\label{P:kleingroup}  Let the indecomposable modules of $V_{4}$ be as above.  Then we have the following containments of ideals.
\begin{enumerate}
\item $\Proj=\ideal_{D}$ is contained in all ideals.
\item For any $n \in \Z$, $\ideal_{k}=\ideal_{\Omega^{n}(k)}$.
\item For any $n \in  \Z_{>0}$ and $m=1, \dotsc ,n-1$, $\ideal_{V_{m}(\infty)}
  \subsetneq \ideal_{V_{n}(\infty)}$.
\item For any $\alpha \in k$, $n \in  \Z_{>0}$, and $m=1, \dotsc ,n-1$, we have
  $\ideal_{V_{m}(\alpha)} \subsetneq \ideal_{V_{n}(\alpha)}$.
\item For $\alpha \in k$, $\alpha \neq 0, 1$, we have
  $\ideal_{V_{2}(\alpha)}=\ideal_{V_{1}(\alpha)}$.
\end{enumerate}  
There are no other nontrivial containments. 
In short: $\forall \alpha\in k\setminus\{0,1\},\forall \beta\in
\{0,1,\infty\},\forall n\in\Z,$
\[
\Proj=\ideal_{D}\subsetneq
\begin{array}[c]{c}
  \ideal_{V_{1}(\alpha)}=\ideal_{V_{2}(\alpha)}\subsetneq
  \ideal_{V_{3}(\alpha)}\subsetneq\cdots \\
  \ideal_{V_{1}(\beta)}\subsetneq\ideal_{V_{2}(\beta)}\subsetneq
  \ideal_{V_{3}(\beta)}\subsetneq\cdots 
\end{array}
\subsetneq \ideal_{\Omega^{n}(k)}=\ideal_{k}=\cat.
\]
\end{proposition}

We now examine the last case of the previous result in greater detail.  Let us fix $\alpha \in k$ with $\alpha \neq 0, 1$.  Fix generators $g,h \in V_{4}$ and set $x = 1+g$ and $y=1+h$ in $kV_{4}$.  Then $kV_{4} \cong k[x,y]/(x^{2}, y^{2})$.  It turns out that it is convenient to work with this presentation of the group algebra.  Then $V_{1}(\alpha)$ is a two dimensional representation with ordered basis $v_{1}, v_{2}$ and the action of $x$ and $y$ are given by the matrices $X_{1}$ and $Y_{1}$ (with respect to the given ordered basis), respectively, where 
\[
X_{1}=\left(\begin{matrix} 0 & 1 \\
                     0 & 0
\end{matrix} \right) \text{ and }
Y_{1}=\left(\begin{matrix} 0 & \alpha \\
                     0 & 0
\end{matrix} \right). 
\]  

We now consider $V_{1}(\alpha) \otimes V_{1}(\alpha)$.  For our purposes it is convenient to use the ordered basis $\{v_{1}\otimes v_{1}, v_{1}\otimes v_{2}+ v_{2}\otimes v_{1}, v_{2}\otimes v_{2}, v_{1}\otimes v_{2} \}$.  For brevity's sake we write $v_{ij}$ for the vector $v_{i}\otimes v_{j}$, where $i,j \in \{1,2 \}$.  In this basis the action of $x$ and $y$ are given by the $4 \times 4$ matrices 
\[
X_{2}=\left(\begin{matrix} 0 & W_{1} \\
                     0 & 0
\end{matrix} \right) \text{ and }
Y_{2}=\left(\begin{matrix} 0 & W_{\alpha} \\
                     0 & 0
\end{matrix} \right), 
\]  where $0$ denotes the $2 \times 2$ zero matrix, and where 
\[
W_{1}= \left(\begin{matrix} 1 & 1 \\
                                   1        & 0
\end{matrix} \right) \text{ and }   W_{\alpha}= \left(\begin{matrix} \alpha^{2} & \alpha \\
                                   \alpha        & 0
\end{matrix} \right).
\] 

A calculation shows that a $4 \times 4$ matrix commutes with $X_{2}$ and $Y_{2}$ and, hence, defines an element of $f \in \End_{\cat}(V_{1}(\alpha) \otimes V_{1}(\alpha))$ if and only if it is of the form 
\[
f= \left(\begin{matrix} A & B \\
                     0 & A^{T}
\end{matrix} \right),
\]  where $A$ and $B$ are $2 \times 2$ matrices of the form 
\begin{equation*}
A = \left(\begin{matrix} a & c \\
                         0 & a 
\end{matrix} \right) \text{ and } B=\left(\begin{matrix} b_{1} & b_{2} \\
                                            b_{3} & b_{4}
\end{matrix} \right)
\end{equation*} with $a,c,b_{1}, b_{2}, b_{3}, b_{4}$ arbitrary elements of $k$ and $A^{T}$ denoting the transpose of $A$.

Given $f$ as above we now compute $\langle \Tr_{R}(f) \rangle$ and $\langle \Tr_{L}(f) \rangle$.  As we did in Section~\ref{SSS:char2} we make use of the following observation: since $V_{1}(\alpha)$ has a simple socle spanned by $v_{1}$, then for any $h \in \End_{\cat}(V_{1}(\alpha))$ we have $h(v_{1}) = \langle h \rangle v_{1}$.  Thus it suffices to compute what happens to the element $v_{1}$.  We first compute $\Tr_{R}(v_{1})$ using the definition:
\begin{align*}
v_{1} &\mapsto v_{11} \otimes v_{1}^{*} + v_{12} \otimes v_{2}^{*} \\
      & \mapsto \left( av_{11}  \right) \otimes v_{1}^{*} + \left(b_{2}v_{11} + b_{4}(v_{12}+v_{21})+av_{12} \right)\otimes v_{2}^{*} \\
      & \mapsto b_{4} v_{1}.
\end{align*} Hence we have
\[
\langle \Tr_{R}(f) \rangle =  b_{4}.
\]

On the other hand, we consider $\Tr_{L}(v_{1})$:
\begin{align*}
v_{1} &\mapsto v_{1}^{*} \otimes v_{11} +  v_{2}^{*}\otimes v_{21} \\
      & = v_{1}^{*} \otimes v_{11} + v_{2}^{*}\otimes \left((v_{12} +v_{21}) + v_{12} \right) \\
       & \mapsto v_{1}^{*}\otimes (a v_{11}) + v_{2}^{*} \left(cv_{11}+a(v_{12}+v_{21}) + b_{2}v_{11}+b_{4}(v_{12}+v_{21}) + av_{12} \right)  \\
       & \mapsto b_{4} v_{1}.
\end{align*}   Hence we have
\[
\langle \Tr_{L}(f) \rangle =  b_{4}.
\] Summarizing the above results, we have the following proposition.

\begin{proposition}\label{P:V1alphaisambi}  Let $k$ be an algebraically closed field of characteristic $2$, let $V_{4}$ be the Klein four group, and let $V_{1}(\alpha)$ be the two dimensional indecomposable $V_{4}$-module where $\alpha \in k$ and $\alpha \neq 0,1$.  

Then the canonical trace on $V_{1}(\alpha)$ is ambidextrous and, hence, $V_{1}(\alpha)$ is ambidextrous.  Furthermore, $V_{2}(\alpha) \in \ideal_{V_{1}(\alpha)}$ and 
\[
\md_{V_{1}(\alpha)}\left(V_{2}(\alpha) \right) =0.
\]  

\end{proposition}

\begin{proof}  The statement that $V_{1}(\alpha)$ is ambidextrous follows immediately from the above calculations.  It follows from the calculations of Conlon that $V_{1}(\alpha) \otimes V_{1}(\alpha) \cong V_{2}(\alpha)$.  Taking $f$ to be the identity map in the above calculation, we immediately obtain the final statement.
\end{proof}

\begin{remark}\label{R:nonambis}  Computer calculations show that $V_{n}(\alpha)$ for $n=2,3,4$ and $\alpha \in k$, $\alpha \neq 1,0$, and $V_{n}(\infty)$ for $n=1,2,3,4$ are not ambidextrous.

\end{remark}

\section{Representations of $\sll_{2}(k)$ In Positive Characteristic}\label{S:sl2}  

\subsection{}\label{SS:LieAlgprelims}Fix an algebraically closed field of characteristic $p>2$.  Let $\fg = \sll_{2}(k)$ and let $\cat$ be the category of \emph{all} finite dimensional $\fg$-modules. The ribbon category structure on $\cat$ is given by the usual coproduct, braiding, etc.   Our standard references for results on $\cat$ is \cite{FP}.  In this section we will examine the simple objects in $\cat$.

Recall that $\fg$ is a restricted Lie algebra and so admits a $p$-power map $x \mapsto x^{[p]}$ ($x \in \fg$) and that in the enveloping algebra, $U=U(\fg )$, the elements $x^{p}-x^{[p]}$ are central.  Let $\mathcal{O}$ denote the central subalgebra of $U$ generated by the set $\{ x^{p}-x^{[p]} \mid x \in \fg \}$.  If $H,E,F$ are the standard basis elements of $\fg$, then $\mathcal{O}$ is a polynomial ring in the elements $H^{p}-H^{[p]}=H^{p}-H$, $E^{p}-E^{[p]}=E^{p}$, and $F^{p}-F^{[p]}=F^{p}$ (recalling that $H^{[p]}=H$, $E^{[p]}=F^{[p]}=0$).  Given a simple $\fg$-module, $S$, then Schur's lemma implies that $S$ is absolutely simple.  Thus every element of $\mathcal{O}$ acts on $S$ by a scalar and there is an algebra homomorphism $\chi : \mathcal{O} \to k$ such that $z.s = \chi(z)s$ for all $z \in \mathcal{O}$ and $s \in S$.  For each such $\chi$, one can consider the full subcategory of $\cat$, $\cat^{[\chi]}$ consisting of all modules which are annihilated by $(z-\chi(z))^{N}$ for all $z \in \mathcal{O}$ and sufficently large $N$.  One then has the following decomposition 
\[
\cat = \bigoplus_{\chi} \cat^{[\chi]},
\] where the direct sum is over all algebra homomorphisms $\chi : \mathcal{O} \to k$.  In particular an indecomposable module lies entirely within one $\cat^{[\chi]}$.  We call each full subcategory $\cat^{[\chi]}$ a \emph{block} of $\cat$.  Also note that if $X \in \cat^{[\chi_{1}]}$ and $Y \in \cat^{[\chi_{2}]}$ then $X \otimes Y \in \cat^{[\chi_{1}+\chi_{2}]}$, and $X^{*} \in \cat^{[-\chi_{1}]}$.

The blocks $\cat^{[\chi]}$ of $\cat$ correspond bijectively to elements of $\fg^{*}$.  Recall that $G=GL_{2}(k)$ acts on $\fg$ via the adjoint action and on $\fg^{*}$ by the dual adjoint action.  Given a $\fg$-module $M$ and $g \in G$, we write $M^{g}$ for the \emph{twist of $M$ by $g$}; that is, the $\fg$-module given by setting $M^{g} = M$ as a vector space with $\fg$-action given by $x.m = (g.x)m$ for all $x \in \fg$ and $m \in M^{g}$.  If two elements of $\fg^{*}$ are conjugate under the dual adjoint action of $GL_{2}(k)$ on $\fg^{*}$, say by $g \in G$, then the corresponding blocks are isomorphic via twisting by $g^{-1}$.  Since twisting by an element of $G$ is a tensor functor which defines an automorphism on the category $\cat$, we can restrict our analysis to moduled which lie in $\cat^{[\chi]}$ for a $\chi$ in each of the orbits of $\fg^{*}$; thus we are reduced to the following cases.  See \cite[Section~2]{FP} for details.
\begin{enumerate}

\item[I.] The \emph{restricted type}.

\begin{tabular}{ccc} 
$\chi(H^{p}-H)=0$ & $\chi(E^{p})=0$ & $\chi(F^{p}) =0.$
\end{tabular}  In this case $\chi =0$ and we write $\chi_{0}$.

\item [II.] The \emph{semisimple type}. For any fixed  $\alpha \in k,$ $\alpha \neq 0$,

\begin{tabular}{ccc}
$\chi(H^{p}-H)=\alpha^{p}$ & $\chi(E^{p})=0$ & $\chi(F^{p}) =0$
\end{tabular} 
\item[III.] The \emph{regular nilpotent type}.

\begin{tabular}{ccc} 
$\chi(H^{p}-H)=0$ & $\chi(E^{p})=0$ & $ \chi(F^{p}) =1$.
\end{tabular}

\end{enumerate}

\subsection{The ideals of $\cat$} We first analyze the ideals generated by a simple module in each of the cases.

\noindent \textbf{Case I:}  Say $\chi=\chi_{0}$ and $J \in \cat^{[\chi_{0}]}$ is simple.   Then by \cite[Proposition~2.4]{FP} one has the following possibilities.
\begin{enumerate} 
\item If $\md_{k}(J) =0$, then $J$ has vector space dimension $p$ and $J$ is the Steinberg module.  But the Steinberg module is projective and then by Lemma~\ref{L:projectives} we have 
\[
\ideal_{J} = \Proj.
\] 
\item One has $\md_{k}(J) = \dim_{\cat}(J) \neq 0$ and by Lemma~\ref{L:dim0} we have
\[
\ideal_{J}=\ideal_{k}=\cat.
\]  

\end{enumerate}

\noindent \textbf{Case II:}  Say $\chi$ is of semisimple type and $J$ is a simple object in $ \cat^{[\chi]}$.   Then by \cite[Corollary~2.2]{FP} $\cat^{[\chi]}$ is a semisimple category and, hence, $J$ is a projective module.  Applying Lemma~\ref{L:projectives} we have
\[
\ideal_{J} = \Proj.
\] 

\noindent \textbf{Case III:} Say $\chi$ is of regular nilpotent type and $J$ is a simple object in $\cat^{[\chi]}$.  By \cite[Proposition~2.3]{FP} we have the following two possibilities.
\begin{enumerate} 
\item If $J$ is the unique simple projective in $\cat^{[\chi]}$, then again we have 
\[
\ideal_{J} = \Proj.
\]
\item If $J$ is simple and not projective, then  $\dim_{k}(J) = p$ and so $\md_{k}(J)=0$.  By Lemma~\ref{L:projectives} and Lemma~\ref{L:dim0} one has 
\begin{equation}\label{E:nilpotentcase}
\Proj \subsetneq \ideal_{J} \subsetneq \ideal_{k}. 
\end{equation}
We now further analyze the objects which appear in $\ideal_{J}$ in this
situation.

By \cite[Proposition 2.3]{FP} the simple modules in $\cat^{[\chi]}$ are
parameterized non-uniquely by elements $\lambda \in k$ which satisfy $\lambda^{p}-\lambda = 0$; that is, the elements of the prime subfield $\mathbb{F}_{p} \subseteq k$. As we will need it later, let us give a brief sketch of the
structure of these modules (we follow the construction in the proof of
\cite[Proposition 2.3]{FP}).  As a point of terminology, given a
$\fg$-module, $M$, we call a vector $m \in M$ a \emph{primitive} vector if
$E.m=0$.  Given $ \lambda \in \mathbb F_p\subset k$.
The simple module $V_{\chi, \lambda}$
is a ``baby Verma module'' and is generated by a primitive vector of weight
$\lambda$, say $v_{0}$.  Then 
\[
\{v_{i}:=F^{i}v_{\lambda} \mid i = 0 ,\dotsc ,
p-1 \}
\]
is a basis for $V_{\chi, \lambda}$. 
On this basis the action of $\sll_{2}$ is given by
\begin{align}\label{E:sl2action}
Hv_{i}&= (\lambda-2i)v_{i} \notag \\
Fv_{i}&=v_{i+1} \\
Ev_{i}&= i(\lambda - i +1)v_{i-1} \notag
\end{align}
where the subscripts are understood to be written modulo $p$.  In particular, the weights of $V_{\chi, \lambda}$ are precisely the elements of $\mathbb{F}_{p} \subseteq k$ and each has multiplicity one.  Moreover, as $V_{\chi, \lambda}$ is defined by induction from a one dimensional representation of the Borel subalgebra, it is a  universal highest weight representation within $\cat^{[\chi]}$.

If $\lambda = p-1$, then
$V_{\chi, \lambda}$ has a unique primitive vector and is the unique projective simple module in $\cat^{[\chi]}$.
For all other $\lambda$, $V_{\chi, \lambda}$
contains a second primitive vector, namely the vector $F^{\lambda^{*}}v_{0}$, where
$\lambda^{*}:=p-2-\lambda$.  However, $V_{\chi, \lambda^{*}}\cong V_{\chi,
  \lambda}$.  Thus if $S$ is a simple module, then the weight of a primitive
vector identifies $S$ up to isomorphism. There are $(p+1)/2$ simple objects in $\cat^{[\chi]}$ (labelled by $0,1, \dotsc , (p-3)/2, p-1$).

The ideals in this case can be handled as an exercise in the rank variety theory introduced by Friedlander and Parshall \cite{FP0}.  This can also be viewed as the first instance of a much more general result which shows that thick tensor ideals of $\cat$ such as these are classified by rank varieties \cite{FPevt}.  Rather than give all the details, let us summarize by saying that the ambiant space for these varieties is the two dimensional variety 
\[
\mathcal{V}_{\sll_{2}(k)} = \left\{x \in \sll_{2}(k) \mid x^{[p]}=0\right\},
\] and that if $J$ is a nonprojective simple module which lies in $\cat^{[\chi]}$ (where $\chi$ is as in Case III), then the variety of $J$ is the line through the origin in $\mathcal{V}_{\sll_{2}(k)}$ given by $kE$.  If $S$ is a nonprojective simple module in a block isomorphic to $\cat^{[\chi]}$ via twisting by $G$, then the variety for $S$ is a line in $\mathcal{V}_{\sll_{2}(k)}$.  Furthermore, if $S$ is a simple module in $\cat$, then $S \in \ideal_{J}$ if and only if the variety of $S$ lies in the variety of $J$.

\end{enumerate}

Let us summarize the above calculations in the following theorem which also accounts for twisting by $G$.

\begin{theorem}\label{T:sl2ideals}  Let $J$ be a simple $\mathfrak{sl}_2(k)$-module lying in $\cat^{[\nu]}$.  Then the following statements are true.

\begin{enumerate}
\item If $\nu=0$ and $J$ is not isomorphic to the Steinberg module, then $\dim_k(J) \neq 0$ and $\ideal_J = \ideal_k=\cat$.
\item If $J$ is projective (i.e.\  $\nu=\chi_{0}$ and $J$ the Steinberg module, $\nu$ is conjugate to the regular nilpotent and $J$ the unique projective simple object in $\cat^{[\nu]}$, or $\nu$ is conjugate to a semisimple type), then $\ideal_J=\Proj$.
\item If $\nu$ is conjugate to the nilpotent regular and $J$ is not projective, then the simple objects in $\ideal_{J}$ are precisely the simple objects in $\cat$ whose rank variety is contained in the rank variety of $J$.
\end{enumerate}

\end{theorem}

\begin{remark}\label{R:sl2remarks} The complete collection of ideals generated by simple objects looks like $$\Proj \subsetneq \ideal_J \subsetneq  \ideal_k,$$ where $J$ ranges over a one parameter family of simple modules corresponding to lines through the origin in $\mathcal{V}_{\sll_{2}(k)}$.
\end{remark}

\subsection{The restricted and projective modules} We note that a simple object $S$ is ambidextrous if and only if its twist $S^{g}$ is ambidextrous for any $g \in G$.  Therefore the following two theorems completely characterize the simple ambidextrous objects in $\cat$.

\begin{theorem}\label{T:sl2ambis}  Let $L$ be a simple object in $\cat^{[\chi]}$.   If $\chi = \chi_{0}$ or if $L$ is projective, then $L$ is ambidextrous.
\end{theorem}

\begin{proof} 
   If $L$ is in
  $\cat^{[\chi_{0}]}$ and is not projective (i.e.\ not the Steinberg
  representation), then the statement follows from Lemma~\ref{P:ambdim} and
  the fact that $\md_{k} (L) \neq 0$.

  Now assume $L$ is projective.  We first explicitly determine an ambidextrous
  projective object.  Fix $\alpha \in k,$ $\alpha \neq 0$ and let $\chi$ be
  the semisimple character corresponding to $\alpha$ in Case II.  By
  \cite[Proposition 2.1]{FP0} the simple objects in $\cat^{[\chi]}$ are
  described as follows.  Let $\lambda \in k$ be an element which satisfies
  $\lambda^{p}-\lambda = 0$; that is, $\lambda$ is an element of the prime subfield $\mathbb{F}_{p} \subseteq k$.
  Then there is a unique irreducible module in
  $\cat^{[\chi]}$ which we will denote by $V_{\chi, \lambda}$ which is
  commonly called the ``baby Verma module'' of highest weight $\lambda$.  It
  is generated by a vector $v_{\lambda}$ satisfying $Hv_{\lambda} = \lambda
  v_{\lambda}$ and $Ev_{\lambda}=0$ and with basis $\{v_{i}:=F^{i}v_{\lambda}
  \mid i = 0 ,\dotsc , p-1 \}$.  The element $\lambda$ uniquely determines
  $V_{\chi, \lambda}$ and the $p$ elements of $\mathbb{F}_{p}$
  give a complete irredundant set of simple objects in $\cat^{[\chi]}$. Now we
  consider the module
\[
T=V_{\chi, \lambda} \otimes V_{\chi, \lambda}.
\]  Then $T$ is an object in $\cat^{[2\chi]}$.  However $2\chi$ is again semisimple (corresponding to the element $2\alpha \in k$), hence $T$ is completely reducible into baby Verma modules of the form $V_{2\chi, \mu}$.  A direct calculation verifies that in fact we have
\[
T \cong V_{2\chi, 2\lambda} \oplus V_{2\chi, 2\lambda-1} \oplus V_{2\chi, 2\lambda-2} \oplus \dotsb  \oplus V_{2\chi, 2\lambda-(p-1)}.
\]  
These simple modules have different
highest weights and, hence, are pairwise non-isomorphic.  By
Remark~\ref{R:ambi} it follows that $V_{\chi, \lambda}$ is ambidextrous.  Now
since $V_{\chi, \lambda}$ is projective, it follows that $\ideal_{V_{\chi,
    \lambda}}=\Proj$ and from Theorems~\ref{T:ambi!Trace}
and~\ref{T:inducedtrace} that we have a (possibly trivial) ambidextrous trace
on any simple $L \in \Proj $.  On the other hand, by Lemma~\ref{L:dim0} we
have
\[
0 \neq \md_{V_{\chi, \lambda}}(L) = \mt _{L}(\Id_{L}),
\] so $\mt_{L}$ is nontrivial.  Finally, since $L$ is simple, $\mt_{L}$ is necessarily a scalar multiple of the canonical trace and hence $L$ is ambidextrous. Therefore every projective simple object in $\cat$ is ambidextrous.  
\end{proof}

\subsection{The nonprojective simples in the nilpotent regular block} We now consider the nonprojective simple objects in $\cat^{[\chi]}$ when
$\chi$ is the nilpotent regular from Case III. 

We first study the decomposition of $V_{\chi,0} \otimes V_{\chi,0}$.  Similar formulas were obtained in \cite{BO, Pr}. We note that if $\chi$ is regular nilpotent, then $2\chi$ is $GL_{2}(k)$-conjugate to $\chi$ and hence our discussion of the category $\cat^{[\chi]}$ applies equally well to $\cat^{[2\chi]}$.

\begin{lemma}\label{L:nilpotenttensors} Let $\chi$ be the regular nilpotent in Case III and let $V=V_{\chi,0}$ be the simple $\fg$-supermodule with primitive vector of weight $0$.  Then 
\begin{equation}\label{E:tensordecomp}
V \otimes V = W_{0} \oplus W_{1} \oplus \dotsb \oplus W_{(p-3)/2} \oplus V_{2\chi, p-1}.
\end{equation}
where $V_{2\chi, p-1}$ is the unique projective simple object of $\cat^{[2\chi]}$.   For $k=1, \dotsc , (p-3)/2$ each $W_{k}$ is isomorphic to the projective cover of $V_{2\chi, k}$.  Furthermore $W_{0}$ is the direct sum of two simple modules, both of whom are isomorphic to $V_{2\chi, 0}$.
\end{lemma}

\begin{proof}  We begin by observing that if $\lambda \in \mathbb{F}_{p}$ and $\lambda \neq 0,p-2$ then the space of primitive vectors of weight $\lambda$ in $V \otimes V$ is one dimensional; if $\lambda = 0$ or $p-2$, then it is two dimensional.  Namely, since $H$ acts semisimply on the space of primitive vectors, we can assume without loss that a primitive vector is a weight vector.  Thus for a fixed $n=0, \dotsc, p-1$ we can consider the equation 
\[
0 = E \left(\sum_{i=0}^{p-1} \alpha_{i} v_{n-i} \otimes v_{i} \right).
\] By direct calculation we see that up to a nonzero scalar there is a unique solution for $\alpha_{0}, \dotsc , \alpha_{p-1}$ if $n \neq 0, 1$.  On the other hand, if $n=0$ or $1$, then we see that the equation has two linearly independent solutions.

Now, given a primitive vector of weight $\lambda$, the universal property of $V_{2\chi, \lambda}$ implies that it appears as a simple module in the socle of $V \otimes V$.  Conversely, any simple module in the socle gives rise to a pair of primitive vectors, $\lambda$ and $\lambda^{*}=p-2-\lambda$.  Hence by our calculation we have 
\[
\operatorname{socle}\left(V \otimes V \right) \cong V_{2\chi, 0} \oplus V_{2\chi, 0} \oplus V_{2\chi, 1} \oplus \dotsb \oplus V_{2\chi, (p-3)/2} \oplus V_{2\chi, p-1}.
\]

We now show that $W_{0}$ is $2p$ dimensional.  Once we do so, dimension considerations show that the decomposition of $V \otimes V$ must be as given.  Namely, $W_{0} \cong V_{2\chi, 0}\oplus V_{2\chi,0}$ and each $W_{1}, \dotsc , W_{(p-3)/2}$ can either be isomorphic to $V_{2\chi, k}$ or its projective cover (which is a nonsplit self extension of the simple module) and by dimensions they are necessarily the projective cover in every case.

Let $\Omega = EF + FE + H^{2}/2 \in U(\sll_{2}(k))$ be the Casimir element of $\sll_{2}(k)$.   We note that if for $a \in \mathbb{F}_{p}$ we set 
\[
c_{a} = a+ \frac{a^{2}}{2},
\] 
then $\Omega$ acts on the simple module $V_{2\chi, a}$ in $\cat^{[2\chi]}$ by
the scalar $c_{a}$.  Furthermore, $c_{a}= c_{b}$ if and only if $a=b$ or $a =
b^{*}$.  That is, $\Omega$ acts on each simple by a unique scalar.

Define 
\[
\Omega_{1,2} = E \otimes F + F \otimes E + (H \otimes H)/2.
\]  We then have that
\begin{equation}\label{E:Omega12}
2\Omega_{1,2}= \Delta(\Omega) - \Omega \otimes 1 - 1 \otimes \Omega
\end{equation} as a linear map on $V \otimes V$,
where $\Delta: U(\sll_{2}(k)) \to U(\sll_{2}(k))$ is the coproduct on the enveloping algebra of $\sll_{2}(k)$. Since the Casimir element is central, from \eqref{E:Omega12} we see that $\Omega_{1,2}$ is in fact a $\sll_{2}(k)$-module endomorphism.  Furthermore, if $V_{2\chi, a}$ is a simple submodule of $V_{2\chi, b} \otimes V_{2\chi, c}$, then $\Omega_{1,2}$ acts on $V_{2\chi, a}$ by the scalar $(c_{a}-c_{b}-c_{c})/2$.  In particular, since in this case $V_{2\chi, b} = V_{2\chi, c}= V_{\chi, 0}$, we have that $\Omega_{1,2}$ acts on $V_{\chi, a}$ by the scalar $c_{a}/2$.  From this it follows that each of the direct summands in \eqref{E:tensordecomp} is precisely a generalized eigenspace of $\Omega_{1,2}$ acting on $V \otimes V$.  Therefore to show $W_{0}$ is $2p$ dimensional we simply need to show that the generalized $c_{0}=0$-eigenspace has dimension $2p$.  Furthermore, since $\Omega_{1,2}$ preserves the weight spaces of $V \otimes V$ and the simple modules in $\cat^{[\chi]}$ have each weight appearing with multiplicity one, this is equivalent to showing that the $p-2$ weight space of the generalized $0$-eigenspace of $\Omega_{1,2}$ is two dimensional.  

Let $z_{i}= v_{i}\otimes v_{1-i}$ for $i=0, \dotsc , p-1$ (where the
subscripts are understood to be written modulo $p$).  Then $z_{0}, \dotsc ,
z_{p-1}$ is a basis for the $p-2$ weight space of $V \otimes V$.  In this
ordered basis we can use \eqref{E:sl2action} to write the matrix $X$ for the
action of $\Omega_{1,2}$ and doing so we obtain the $p \times p$ matrix
\[
X=\left(
  \begin{matrix} 
    0 & 0 & 0 & 0 & \dotsb & 0 & -2 \\
    0 & 0 & -2 & 0 & \dotsb & 0 & 0 \\
    0 & 0 & &&&& \\
    \vdots & \vdots & & & M & & \\
    0 & 0 &&& &&
  \end{matrix} 
\right), \hspace{.5in} N =
\left(\begin{matrix}2&1&&&&\\1&2&1&&&\\
    &1&2&1&\\ &&1&2&& \\ &&&\ddots&\ddots& \\ &&&&&1 \\ &&&&1&2\end{matrix} \right),
\] where $M$ and $N$ are $(p-2)\times (p-2)$ matricies,  $M=ND$, and $D$ is the invertible diagional matrix $D=\operatorname{diag}((i(1-i))_{i=2,...,p-1})$.
The result then follows from the fact that for any $n\in\N$,
$\rank(X^n)=\rank(M^n)=p-2$.  Indeed $M$ is invertible.  This follows from the fact that $D$ is invertible and an easy
inductive argument which shows that $\det(N)=p-1$ (alternatively, this determinant
is given by $U_{p-2}(1)=p-1$, where $U_{n}(x)$ is the degree $n$ Chebyshev
polynomial of the second kind).
\end{proof}
Using this lemma we now consider the special case of $V_{\chi, 0}$. 
\begin{proposition}\label{P:Vzero}  Let $\chi \in \fg^{*}$ be regular nilpotent and let $V_{\chi, 0}$ be the simple module in $\cat^{[\chi]}$ labelled by $0 \in \mathbb{F}_{p}$.  Then $V_{\chi, 0}$ is not ambidextrous.
\end{proposition}

\begin{proof} Let $V=V_{\chi, 0}$.  Since the square of the braiding is the
  identity and the characteristic of $k$ is not equal to two, by
  Remark~\ref{R:ambi} it suffices to prove that there always exists a morphism
  $f: V \otimes V \to V\otimes V$ which is between $S^{2}V$ and $\Lambda^{2}V$
  and for which $\Tr_{R}(f) \neq 0$.

  We first observe that since $v_{0}, v_{1} \in V$ are the primitive vectors
  of weight $0$ and $p-2$, respectively, the vectors $v_{0}\otimes v_{1}+
  v_{1}\otimes v_{0}$ and $v_{0}\otimes v_{1}- v_{1}\otimes v_{0}$ are both
  primitive vectors in $V \otimes V$ of weight $p-2$.  Therefore, from
  \eqref{E:tensordecomp} we see that $S^{2}V$ and $\Lambda^{2}V$ each have a
  direct summand which is isomorphic to $V_{2\chi,0}$.  Let us call $W'$
  (resp.\ $W''$) the summand which lies in $S^{2}V$ (resp.\ $\Lambda^{2}V$).
  We then have a morphism $f: V \otimes V \to V \otimes V$ which is
  identically zero on all direct summands of \eqref{E:tensordecomp} except
  on $W'$ which is mapped isomorphically to $W''$; in
  particular, it is determined by setting
  \[
  f(v_{0}\otimes v_{1})=f(v_{1}\otimes v_{0})=v_{0}\otimes v_{1}- v_{1}\otimes
  v_{0}.
  \]
  We use the notation of the proof of Lemma \ref{L:nilpotenttensors}.  The
  operator $\Omega_{1,2}$ is zero on $W''$ because it is zero on $v_{0}\otimes
  v_{1}-v_{1}\otimes v_{0}$.  Thus $\Omega_{1,2}\circ
  f=f\circ\Omega_{1,2}=0$.  This and the form of $X$ implies that
  $f(z_i)=-(-1)^i(v_{0}\otimes v_{1}- v_{1}\otimes v_{0})$ for $i=2,\ldots,p-1$.

  Furthermore $f\circ\Omega_{1,2}(v_0\otimes
  v_{i})=i(1-i)f(v_{1}\otimes v_{i-1})=0$ and we get that $f(v_{1}\otimes v_{i})=0$
  except for $i\in\{-1,0\}$.

  As $\Tr_{R}(f)=x\Id_V$ for some $x \in k$, we can now compute $x$ as follows 
  $$
  \Tr_{R}(f)(v_1)=\sum_i (\Id\otimes v_i^*)\big(f(v_1\otimes v_i)\big)
  =(\Id\otimes v_0^*)\big(f(v_1\otimes v_0)\big)+(\Id\otimes
  v_{-1}^*)\big(f(v_1\otimes v_{-1})\big).
  $$
  Now, $(\Id\otimes v_0^*)\big(f(v_1\otimes v_0)\big)=-v_1$, and as
  $F^p$ acts by $2$ on $V\otimes V$, we have 
  $$
    2 f(v_1\otimes v_{-1}) =F^{p-1}f(F(v_1\otimes v_{-1}))=F^{p-1} f(z_2+z_1)=0.
  $$
  Hence $\Tr_{R}(f)=-\Id_V\neq0$ so $V_{\chi, 0}$ is not ambidextrous.
\end{proof}

The general case now follows easily.

\begin{theorem}\label{T:sl2nonambis} If $L$ is a nonprojective simple object
  in $\cat^{[\chi]}$ and $\chi$ is nilpotent regular, then $L$ is not
  ambidextrous.
\end{theorem}

\begin{proof} The key observation is that if some nonprojective simple object
  in $\cat^{[\chi]}$ is ambidextrous, then using Theorems~\ref{T:ambi!Trace}
  and~\ref{T:inducedtrace} and Lemma~\ref{L:dim0}, we would have that
  \emph{all} nonprojective simple objects in $\cat^{[\chi]}$ are ambidextrous.
  However, by the previous proposition this is not the case.
\end{proof}

\linespread{1}

\vfill

\end{document}